\documentclass[11pt,a4paper, parskip=half]{scrartcl} 
\usepackage[utf8]{inputenc}
\usepackage[english]{babel}
\usepackage[T1]{fontenc}
\usepackage{lmodern}
\usepackage[left=3cm,right=3cm,top=3cm,bottom=3cm]{geometry}
\usepackage[runin]{abstract}
    \abslabeldelim{.}

\usepackage{todonotes}

\usepackage{amsmath, amsfonts, amssymb, amsthm, amstext}
\usepackage{mathtools}
\usepackage{mathrsfs}
\usepackage{latexsym}
\usepackage{mathdots}
\usepackage{cases}

\usepackage{graphicx}
\usepackage{subcaption}
\usepackage{algorithm}

 at11pt
 at9pt

\usepackage{enumitem}

\usepackage{color}

\usepackage{hyperref} 
\usepackage[capitalise]{cleveref}
	
\newcounter{thm}
\numberwithin{thm}{section}
\numberwithin{equation}{section}

	\newtheoremstyle{myplain}		
			{}			
			{}			
			{\itshape}				
			{}				
			{\sffamily\bfseries}				
			{.}		
			{ }				
			{\thmname{#1}\thmnumber{ #2}\textnormal{\textsf{\thmnote{ (#3)}}}}			
    \newtheoremstyle{mybreak}
            {}{}{}{}{\sffamily\bfseries}{.}{\newline}
            {\thmname{#1}\thmnumber{ #2}\textnormal{\textsf{\thmnote{ (#3)}}}}
	\newtheoremstyle{mydef}
			{}{}{}{}{\sffamily\bfseries}{.}{ }
			{\thmname{#1}\thmnumber{ #2}}
	\newtheoremstyle{myrem}
			{}{}{}{}{\sffamily\itshape}{.}{ }
			{\thmname{#1}\thmnumber{ #2}}

\theoremstyle{myplain}
	\newtheorem{theorem}[thm]{Theorem}

\theoremstyle{mybreak}
\theoremstyle{mydef}
	\newtheorem{definition}[thm]{Definition}
	\newtheorem{remark}[thm]{Remark}
\theoremstyle{mydef}

	\newcommand{\cc}{\mathbb{C}}
		
		\newcommand{\nn}{\mathbb{N}}
	\newcommand{\rr}{\mathbb{R}}

\allowdisplaybreaks

\def\sumprime_#1^#2{
    \setbox0=\hbox{$\scriptstyle{#1}$}
    \setbox1=\hbox{$\scriptstyle{#2}$}
    \setbox2=\hbox{$\displaystyle{\sum}$}
    \setbox4=\hbox{${}^\prime\mathsurround=0pt$}
    \dimen0=.5\wd0 \advance\dimen0 by-.5\wd2
    \ifdim\dimen0>0pt
        \ifdim\dimen0>\wd4 \kern\wd4
        \else\kern\dimen0
        \ifdim\dimen1>\wd4 \kern\wd4
        \else\kern\dimen1
    \fi\fi\fi
\mathop{{\sum}^\prime}_{\kern-\wd4 #1}^{\kern-\wd4 #2}
}
\title{\Large  Neural network-based singularity detection \\ 
and applications}
\author{Nadiia Derevianko\footnote{TUM School of CIT, 
Department of Computer Science,
Boltzmannstrasse 3,
85748 Garching b. München,
Germany,  nadiia.derevianko@tum.de} \footnote{Corresponding author}, 
Ioannis G. Kevrekidis \footnote{Departments of Chemical and Biomolecular Engineering   and Applied Mathematics and Statistics, Johns Hopkins University, Baltimore, MD, USA,  yannisk@jhu.edu}, 
Felix Dietrich \footnote{TUM School of CIT, MDSI, \& MCML,
Department of Computer Science,
Boltzmannstrasse 3,
85748 Garching b. München,
Germany,  felix.dietrich@tum.de} }

\date{\today}

\begin{document}

\maketitle

\begin{abstract}
We present a method for constructing a special type of shallow neural network that learns univariate meromorphic functions with pole-type singularities. Our method is based on using a finite set of Laurent coefficients 
as input information, which we compute by FFT, employing values of the investigated function on some contour $\Gamma$ in the complex plane. 
The primary components of our methodology are the following: (1) the adaptive construction of rational polynomial 
activation functions, (2) a novel backpropagation-free method for determining the weights and biases of the hidden layer, and (3) the computation of the weights and biases of the output layer through least-squares fitting. 
Breaking with the idea of ``safe'' rational activation functions, we introduce a rational activation function as a meromorphic function with a single pole situated within the domain of investigation. Employing the weights and biases of the hidden layer, we then scale and shift the pole of the activation function to find the estimated locations of the singularities; this implies that the number of neurons in the hidden layer is determined by the number of singularities of the function that is being approximated. 
While the weights and biases of the hidden layer are tuned so as to capture the singularities, the least-squares fitting for the computation of weights and biases of the output layer ensures approximation of the function in the rest of the domain. Through the use of Laurent-Padé rational approximation concepts, we prove locally uniform convergence of our method. We illustrate the effectiveness of our method through numerical experiments, including the construction  of  extensions of the time-dependent solutions of nonlinear autonomous PDEs into the complex plane, and study the dynamics of their singularities.


\textbf{Mathematics Subject:} 	30D30, 	30E10, 41A20,  65D15
\end{abstract}

\section{Introduction}

In this paper, we discuss a method to construct parameters of neural networks that approximate functions $f:\mathbb{C}\to\mathbb{C}$ with pole-type singularities. The main idea is to choose a rational function as approximant of $f$, and to reformulate this approximant into a neural network with a rational activation function and neurons arranged in a single hidden layer (see Figure \ref{nns}). We introduce a new approach to explicitly compute the parameters of these neurons from the parameters of the rational approximant, without iterations, gradient descent, or backpropagation. We then linearly combine the neurons with parameters of the output layer, computed by minimizing a least-squares problem. As a method for rational approximation, we choose the Laurent-Pad\'{e} method because this way the location of pole-type singularities can be estimated. The Laurent-Pad\'{e} method also allows us to prove locally uniform convergence of the neural network approximants.

The question of detection of singularities and approximation of functions with singularities has been considered over many years and by many authors. We refer the interested reader to  \cite{EF2015, L2018, K17,  M92}   for more details  regarding different approaches, and especially to \cite{CC20,FH19,R18,W03,W22} regarding methods based on rational  Laurent-Padé approximation \cite{B87, BR2015, GGT13}.

Rational functions have been also shown to be an effective tool in signal processing \cite{BG24, KG15}, in developing of univariate \cite{BC2020, DPP21} and multivariate \cite{DH25} methods for exponential analysis, in approximating spectral filters in  graph convolutional networks \cite{CCL2018}, etc. Rational functions are also employed as so-called ``adaptive activation functions'' in the definition of neural networks  \cite{BNT20, MSK, Tel17} for learning functions $f: \rr^d \rightarrow \rr$. These rational activation functions are defined as ``safe'' Padé
Activation Units (PAUs) 
$
r(x)=\sum_{j=0}^m \alpha_j x^j  \bigg/ \left( 1+\left| \sum_{j=1}^n \gamma_j x^j \right| \right),
$
preventing the occurrence of poles and enabling safe computation on $\rr$. Here, the coefficients $\alpha_j, \gamma_j \in \rr$ 
are trainable parameters, determined together with the weights and biases of the neural network
by backpropagation and a gradient descent optimization algorithm. A good choice for initialization parameters of the activation functions is obtained when $r(x)$ are constructed as best uniform rational approximations of type $(m,n)$ to the ReLU function in $[-1,1]$, i.e. the activation functions are continuous.  
An advantage of PAUs is that they enhance the predictive performance of neural networks and offer greater flexibility. In \cite{Tel17, BNT20} is presented the theoretical benefit of using neural networks based on rational activation functions due to their superiority over ReLU in approximating functions. In \cite{BNT20, MSK}, the authors also demonstrate that rational neural networks outperform ReLU-based neural networks for the solutions of certain PDEs and for the classification tasks of various datasets through numerical experiments.

In \cite{P22}, the author developed a new training method for rational neural networks. 
She considered a neural network $\Psi: \rr^d \rightarrow \rr$ (in modified form without output layer) 
given by $\Psi(\boldsymbol{x})=r(\mathbf{W} \boldsymbol{x} - b )$  with  degree-one
rational activation function
$
    r(x)=\frac{\alpha_0+\alpha_1 x}{\gamma_0+\gamma_1 x}$ and $\gamma_0+\gamma_1 x>0$.
 Next, the loss function was defined as follows
$$
\mathcal{L}(\mathbf{W},b)=\max\limits_{k=1,\dots,N}\left|y^{(k)}- \frac{\alpha_0+\alpha_1 (\mathbf{W} \boldsymbol{x}^{(k)} - b )}{\gamma_0+\gamma_1 (\mathbf{W} \boldsymbol{x}^{(k)} - b )}\right|,
$$
where $\gamma_0+\gamma_1 (\mathbf{W} \boldsymbol{x}^{(k)} - b )>0$ for each $k=1,\dots,N$ and $(\boldsymbol{x}^{(k)},y^{(k)}) \in \rr^{d \times 1}$, $k=1,\dots,N$ is the training data set. Finally, the weights $\mathbf{W} \in \rr^{d \times 1}$ and the bias $b \in \rr$ of the network were computed as decision variables of the rational approximation problem. As methods for rational approximation, the  author of \cite{P22} chose the bisection method and the differential correction algorithm. 
Further, this new approach was applied to the binary classification task of the TwoLeadECG dataset from the PhysioNet database  \cite{GP2000}.
Numerical experiments suggested  that the proposed method  gives better classification accuracy when the training dataset is either very small or if the classification classes are imbalanced. 

\begin{figure}[h!]
  \centering
    \includegraphics[width=0.25\linewidth]{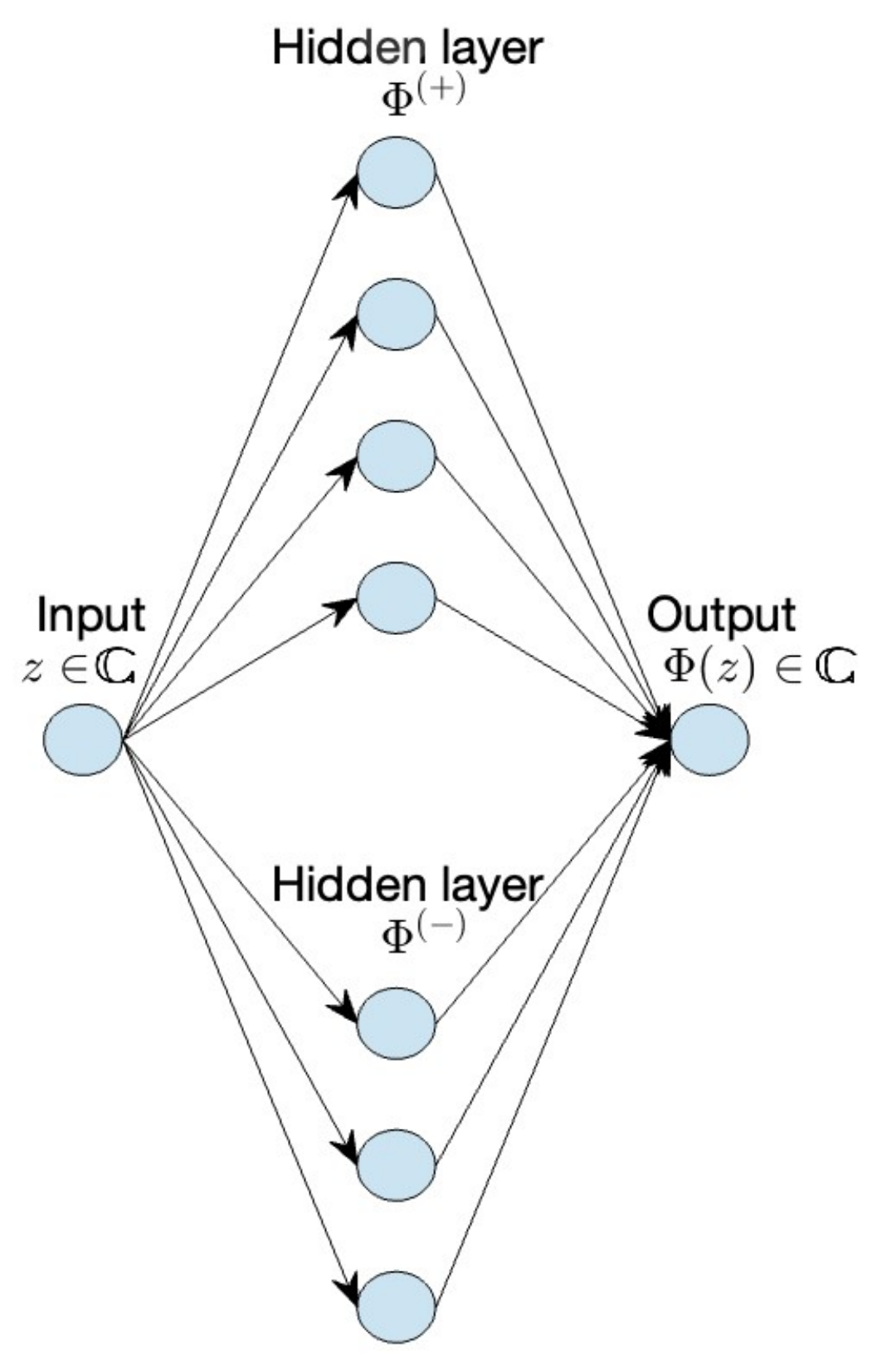}
      \caption{Diagram of the neural network $\Phi$ as in (\ref{nn}) for the case when the function $f$ under consideration has 4 singularities located in the annulus $\Tilde{\mathcal{A}}^{(+)}$ that are detected by the neural network component $\Phi^{(+)}$ and 3 singularities located in the annulus $\Tilde{\mathcal{A}}^{(-)}$ that are detected by the neural network component $\Phi^{(-)}$.}
  \label{nns}
\end{figure}

To our knowledge, a construction of a rational neural network with ``active poles''   (poles located within the investigated domain) has not been presented yet. We relax
the restriction $\gamma_0+\gamma_1 x>0$ given in safe PAUs by controlling the pole $-\gamma_0/\gamma_1$ of the activation function (we consider activation functions with a single pole each) by weights and biases of the hidden layer,  
and develop a neural network for  approximation of functions with pole-type singularities. Our neural network is a discontinuous function itself and, in the limiting case (see convergence Theorem \ref{cont}), singularities of the neural network capture singularities of the function under consideration.
 We also would like to draw attention to the paper \cite{JK2020}, which utilized physics-informed neural networks with continuous trainable activation functions to approximate jump discontinuities of discontinuous functions. 
Numerical results demonstrate that the neural network is capable of accurately capturing the jump locations; however, the neural network approximant is a continuous function itself.

The exact class of functions $\mathcal{M}$ (see Definition \ref{defm}) with pole-type singularities that we want to approximate are functions $f:\mathbb{C}\to\mathbb{C}$ of a single complex variable analytic in the annulus $\mathcal{A}:=\{z \in \mathbb{C}: r< |z|< R \}$ and with meromorphic extension in the annulus  $\Tilde{\mathcal{A}}:=\{z \in \mathbb{C}: \Tilde{r}< |z|< \Tilde{R} \}$ for some $0<\Tilde{r}<r<R<\Tilde{R}<\infty$.  For $f\in \mathcal{M}$, 
we construct a neural network $\Phi:\mathbb{C}\to\mathbb{C}$ in the form $\Phi(z)=\Phi^{(+)}(z)+\Phi^{(-)}(z)$ as in (\ref{nn}), where the neural network component  $\Phi^{(+)}$ is responsible for approximation of $f^{(+)}$ in the set  $\{z \in \mathbb{C}: |z|< R \}$  and detects singularities in the annulus $\Tilde{\mathcal{A}}^{(+)}:=\{z \in \mathbb{C}: R\leq |z|< \Tilde{R} \}$, and the component $\Phi^{(-)}$  approximates $f^{(-)}$ in the set $\{z \in \mathbb{C}: |z|> r \}$ and detects singularities in the annulus  $\Tilde{\mathcal{A}}^{(-)}:=\{z \in \mathbb{C}: \Tilde{r}< |z|\leq r \}$. Note also that it is customary in machine learning to consider neural networks with real-valued parameters, while in our case all parameters that determine $\Phi$ as well as its inputs and outputs are complex. This is required by the application of the Laurent-Pad\'{e} method for rational approximation. For more information on various methods for complex neural networks, we refer the interested reader to \cite{CL22, L22}.

Our methodology is comparable to that of \cite{P22} in the way that both methods involve the reformulation of rational approximants into neural networks with rational activation functions and the utilization of parameters of rational approximants to calculate the weights and biases of the hidden layer(s). There are, however, two significant distinctions. The first is the application of various techniques for  rational approximation. For good approximation of continuous functions and convergence in the maximum norm, the author of \cite{P22} uses the bisection method and the differential correction algorithm. We, on the other hand, apply the Laurent-Pad\'{e} method, which is a direct computation that ensures locally uniform convergence and includes detection of pole-type singularities. Another difference is that in \cite{P22}, the author uses a \textit{continuous} degree-one rational activation function, i.e. a safe PAU. We allow our activation functions to have a single pole each in the domain under investigation, with the goal to further employ this pole to detect the approximate locations of singularities. Note also that we are pioneering in this paper by introducing a method based on \textit{``unsafe''} Padé Activation Units, as we are not aware of any other paper in this regard.

 \begin{figure}[h!]
  \centering
    \includegraphics[width=0.75\linewidth]{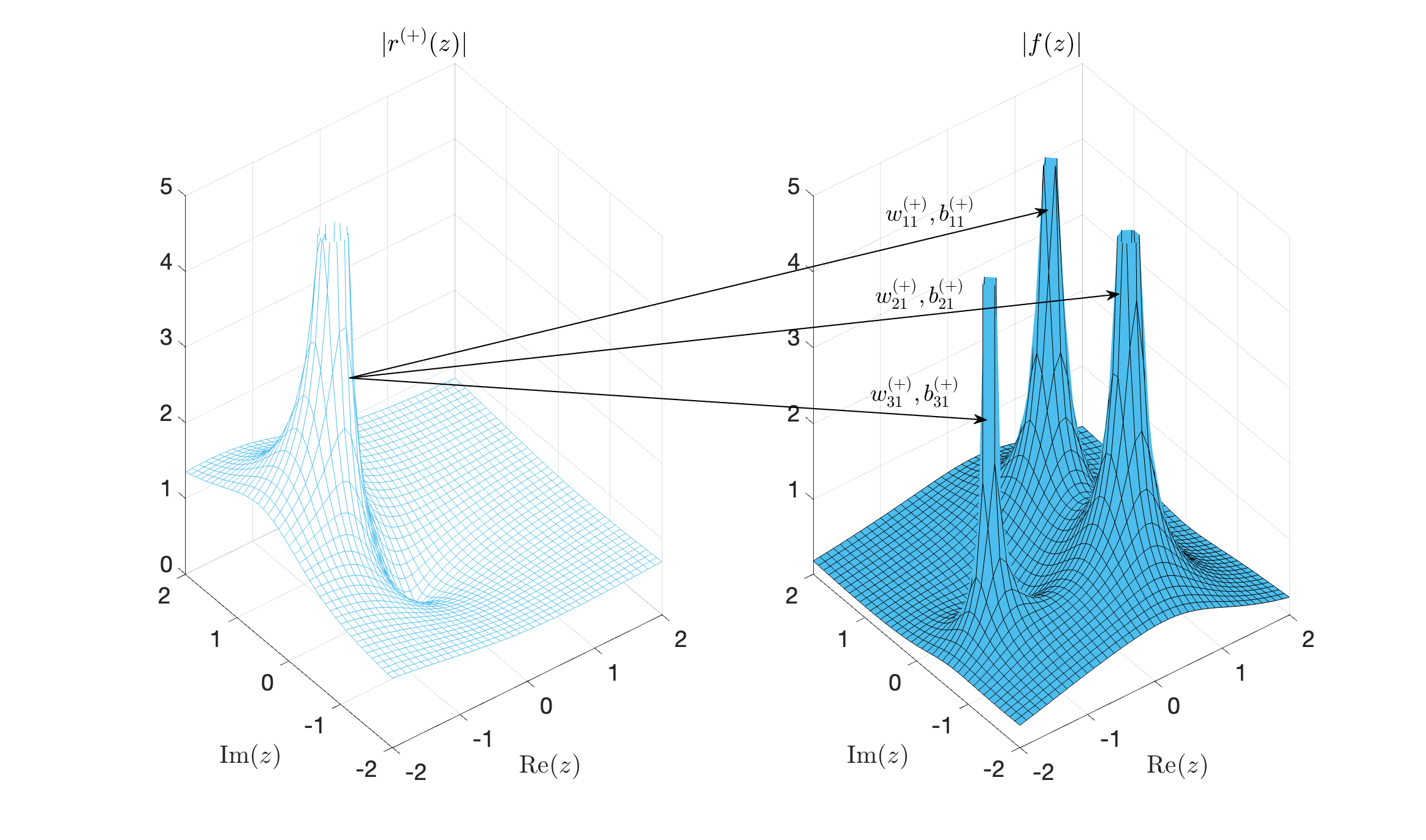}
      \caption{Left: an activation function $r^{(+)}(z)$ as a meromorphic function  with a single pole, right: a function $f(z)$ to approximate, with 3 poles in the annulus $\Tilde{\mathcal{A}}^{(+)}$. The estimated location of each singularity is detected by scaling and shifting the pole of the activation function by weights and biases, respectively, of the hidden layer of the component $\Phi^{(+)}$.}
  \label{figmain}
\end{figure}

According to our method, components $\Phi^{(\pm)}$ defined by (\ref{f1})--(\ref{f2}) are constructed separately.  
Our algorithm utilizes a finite set of Laurent coefficients of the investigated function $f \in \mathcal{M}$ as input data. These coefficients are determined through FFT, using values of $f$ on some contour $\Gamma$ in the complex plane. The choice of the activation function $r^{(\pm)}$ as in (\ref{r1}) is a crucial component of our method.
 We present an adaptive construction of the activation functions $r^{(\pm)}$  as rational Laurent-Pad\'{e} approximations of some special type $(N^{(\pm)}+1-M^{(\pm)},1)$  to $\frac{\varphi^{(\pm)}(z)}{z-z_0^{(\pm)}}$, where $\varphi^{(\pm)}(z)$ are analytic in $ \cc$ and poles $z_0^{(\pm)} \in \cc$ are located outside the unit circle.   Thus, our activation functions are meromorphic functions themselves, with one singularity each. Note that it is also possible to use the same activation function $r^{(\pm)}=r$ for  both  components $\Phi^{(\pm)}$. Parameters $M^{(\pm)}$ determine the numbers of the neurons in the hidden layers of $\Phi^{(\pm)}$, which also corresponds to the numbers of singularities in the annuli $\Tilde{\mathcal{A}}^{(\pm)}$, i.e. our method is based on the idea of \textit{``one singularity, one neuron''}  (see Figure \ref{nns}).

Weights and biases of the hidden layers of the components $\Phi^{(\pm)}$ are determined via our new backpropagation-free approach. Scaling and shifting the poles of the activation functions by weights and biases, respectively, of the hidden layers  to the estimated locations of the poles of the function under consideration is the method by which singularities are detected (see Figure   \ref{figmain}).
   Weights and biases of the output layers  are computed via least-squares fitting, which ensures a good approximation in the rest of the domain. The key idea in the construction of our new neural network is the application of the Laurent-Pad\'{e} method for rational approximation to determine the weights and biases of the components $\Phi^{(\pm)}$. Using this method for rational approximation ensures a good estimation of the location of the singularities and locally uniform convergence of the neural network approximant.

We would also like to  mention very interesting similarities of our proposed approach to another backpropagation-free method, called SWIM (``sample where it matters''), for learning of continuous functions \cite{swim}. The weights and biases of the hidden layer according to SWIM are defined such that we ``learn at the location of largest gradients'', while weights and biases of the hidden layers of  $\Phi^{(\pm)}$ according to our method are determined to capture the location of  singularities of the function that is being approximated. To reduce the number of neurons, in \cite{swim} the authors employ probabilistic arguments and randomly subsample neurons, while for our method to reduce the number of neurons we apply an SVD-based method from \cite{GGT13}.   The weights and biases of the output layers for both methods are determined as solutions to a least-squares problem. Further, in \cite{DD24}, the SWIM method is applied to the solution of time-dependent PDEs with continuous solutions. In this paper, we also present application of our method to the construction of  extensions of the time-dependent solutions of nonlinear PDEs into the complex plane.

\textbf{Outline.}   In Section \ref{le}, we provide the known facts from the theory of meromorphic functions that we  use in the construction of our new neural network $\Phi$. The key concepts of our novel approach are presented in Section \ref{spt}. Namely, in Subsection \ref{ss1},  we discuss ideas for the adaptive  construction of the activation functions $r^{(\pm)}$.  A new backpropagation-free method for the computation of weights and biases of the hidden layers of the neural network components $\Phi^{(\pm)}$ is the subject of Subsection \ref{ss3}. In Subsection \ref{ss4}, we present a method for computing the weights and biases of the output layers of the components $\Phi^{(\pm)}$. We prove the convergence of our neural network approximant in Subsection \ref{cnn}.  In Section \ref{rann}, we present the main steps of the robust numerical method for the construction of the corresponding neural network $\Phi$  and summarize them in Algorithm \ref{alg3}. Finally, in Section \ref{nex}, we demonstrate the efficiency of our novel approach through numerical experiments,  including application to the construction of  extensions of the time-dependent solutions of nonlinear PDEs into the complex plane and to study the dynamics of their singularities.

\textbf{Notation.} As usual, $\nn$, $\rr$ and $\cc$ denote the natural,  real and complex numbers, respectively. Throughout the paper,  we  use the matrix notation $\mathbf{X}=(X_{ij})_{i,j=1}^{m,n} \in \cc^{m \times n}$ for matrices of size $m\times n$, and the vector notation $\boldsymbol{x}=(x_j)_{j=1}^n \in \cc^{ n}$ or $\boldsymbol{x}=(x_1,\ldots,x_n)^T \in \cc^{ n}$ for column vectors of size $n$.  Notation $X^{(\pm)}$ means that we consider two cases, $X^{(+)}$ and $X^{(-)}$.  When we state that $X^{(\pm)}$ are valid for $Y^{(\pm)}$, it means that $X^{(+)}$ is valid for $Y^{(+)}$ and $X^{(-)}$ is valid for $Y^{(-)}$.

\section{Laurent expansion of meromorphic functions}
\label{le}

We now define a class $\mathcal{M}$ of the meromorphic functions under consideration in the paper (see Subsection \ref{ex1} for an example of a function from the class $\mathcal{M}$).
\begin{definition} 
\label{defm}
We say that a function $f:\mathbb{C}\to\mathbb{C}$ belongs to a class $\mathcal{M}$, if $f$ is analytic in the annulus $\mathcal{A}=\{z \in \mathbb{C}: r< |z|< R \}$ and $f$ allows a meromorphic extension in the annulus  $\Tilde{\mathcal{A}}=\{z \in \mathbb{C}: \Tilde{r}< |z|< \Tilde{R} \}$ for some $0<\Tilde{r}<r<R<\Tilde{R}<\infty$. The function $f$ has pole-type singularities located in the annuli $\Tilde{\mathcal{A}}^{(+)}=\{z \in \mathbb{C}: R\leq |z|< \Tilde{R} \}$ and $\Tilde{\mathcal{A}}^{(-)}=\{z \in \mathbb{C}: \Tilde{r}< |z|\leq r \}$. We denote by $s_\ell^{(\pm)}$, $\ell=1,\dots,\kappa^{(\pm)}$ pairwise distinct singularities of $f$ located in the annuli   $\Tilde{\mathcal{A}}^{(\pm)}$. Let $m_\ell^{(\pm)}$ be the multiplicities of $s_\ell^{(\pm)}$, and let the total multiplicities be $M^{(\pm)}:=\sum_{\ell=1}^{\kappa^{(\pm)}} m_\ell^{(\pm)}$ with $M^{(+)}>0$ and/or $M^{(-)}>0$.
\end{definition}

Let $f \in \mathcal{M}$. Then, for $z \in \mathcal{A}$, we represent the Laurent series of $f$ in the form
\begin{align}
     f(z) &=\sum\limits_{k=-\infty}^\infty c_k(f) z^k=f^{(-)}(z)+f^{(+)}(z), \label{ls}
\end{align}
where $f^{(-)}(z):=\sum_{k=-\infty}^{-1} c_k(f) z^k+c_0(f)/2$ is analytic in the set $\{z \in \mathbb{C}: |z|> r \}$ and its singularities are located in the annulus $\Tilde{\mathcal{A}}^{(-)}$, and $f^{(+)}(z):=c_0(f)/2+\sum_{k=1}^\infty c_k(f) z^k$ is analytic in the set $\{z \in \mathbb{C}: |z|< R \}$ and its singularities are located in the annulus $\Tilde{\mathcal{A}}^{(+)}$. By $c_k(f)$ in (\ref{ls}) we denote the Laurent coefficients of $f$ that are given via the Cauchy integral formula,
\begin{equation}\label{lc}
c_k(f)=\frac{1}{2\pi \mathrm{i}} \oint\limits_\Gamma f(z) z^{-k-1} \, \mathrm{d} z,
\end{equation}
where $\Gamma$ is a regular and positively oriented simple loop (closed contour) around $0$ and $f$ is holomorphic on $\Gamma$. We employ these coefficients $c_k(f)$ as input data for our method, which is described in Algorithm  \ref{alg3}. In this paper, we compute $c_k(f)$  by Algorithm \ref{alg2} using values of $f$ at equidistantly distributed points on $\Gamma$.

\section{Neural network for detection of pole-type singularities}
\label{spt}

In this section, we present a construction of our new neural network $\Phi: \cc\rightarrow \cc$ that approximates $f \in \mathcal{M}$ and detects its singularities in the annulus $\Tilde{\mathcal{A}}$. To achieve this property, we construct our neural network in the form
\begin{equation}\label{nn}
\Phi(z)=\Phi^{(+)}(z)+\Phi^{(-)}(z),
\end{equation}
where $\Phi^{(+)}$ is the neural network component that approximates $f^{(+)}$ in the set $\{z \in \mathbb{C}: |z|< R \}$  and detects singularities in the annulus $\Tilde{\mathcal{A}}^{(+)}=\{z \in \mathbb{C}: R\leq |z|< \Tilde{R} \}$, and $\Phi^{(-)}$ is the neural network component that approximates $f^{(-)}$ in the set $\{z \in \mathbb{C}: |z|> r \}$ and detects singularities in the annulus  $\Tilde{\mathcal{A}}^{(-)}=\{z \in \mathbb{C}: \Tilde{r}< |z|\leq r \}$ (see~\Cref{nns}). The two components of (\ref{nn}), $\Phi^{(+)}$ and $\Phi^{(-)}$, are given by 
\begin{align}
   \Phi^{(+)}(z): & =  \mathbf{W}_2^{(+)} r^{(+)} \left( \mathbf{W}_1^{(+)} z - \boldsymbol{b}_1^{(+)}  \right) - b_2^{(+)}, \label{f1} \\
      \Phi^{(-)}(z): & =  \mathbf{W}_2^{(-)} r^{(-)} \left( \mathbf{W}_1^{(-)} z^{-1} - \boldsymbol{b}_1^{(-)}  \right) - b_2^{(-)}, \label{f2}
\end{align}
respectively, where $r^{(\pm)}$ are the activation functions, $M^{(\pm)}$ are the numbers of neurons in the hidden layers of $\Phi^{(\pm)}$,  $\mathbf{W}_1^{(\pm)} \in \cc^{M^{(\pm)} \times 1}$, $ \boldsymbol{b}_1^{(\pm)} \in \cc^{M^{(\pm)}}$ are weights and biases, respectively, of the hidden layers, and $\mathbf{W}_2^{(\pm)}  \in \cc^{1 \times M^{(\pm)} }$, $b_2^{(\pm)} \in \cc$ are weights and biases, respectively, of the output layers. Our further goal is to present methods for the computation of all parameters that determine the  components $\Phi^{(\pm)}$.


\subsection{Adaptive activation functions}
\label{ss1}

In this subsection, we explain the selection of the activation functions $r^{(\pm)}$ for the components $\Phi^{(\pm)}$ of our neural network (\ref{nn}). Let us assume that  $N^{(\pm)}, M^{(\pm)} \in \nn$ are given. In Algorithm \ref{alg1}, we also present a numerical method for computation of these parameters $N^{(\pm)}$ and $M^{(\pm)}$. Note that $M^{(\pm)} \in \nn$ correspond to the numbers of singularities of $f$ in the  annuli $\Tilde{\mathcal{A}}^{(\pm)}$ and  $N^{(\pm)}$ are responsible for the convergence of the neural network $\Phi$. By $M^{(\pm)}$ we also determine the numbers of neurons in the hidden layers of the components $\Phi^{(\pm)}$.

We propose  an adaptive construction of the activation functions $r^{(\pm)}$ as rational functions of certain types $(N^{(\pm)}+1-M^{(\pm)},1)$, i.e.
\begin{equation}\label{r1}
r^{(\pm)}(z):=\frac{\alpha^{(\pm)}_{N^{(\pm)}+1-M^{(\pm)}} z^{N^{(\pm)}+1-M^{(\pm)}} +\dots+\alpha_1^{(\pm)} z+\alpha_0^{(\pm)}}{\gamma_1^{(\pm)} z +\gamma_0^{(\pm)}},
\end{equation}
where $\alpha_j^{(\pm)} \in \cc$, $j=0,\ldots, N^{(\pm)}+1-M^{(\pm)}$ and $\gamma_1^{(\pm)}, \, \gamma_0^{(\pm)} \in \cc$.
Based on numerical experiments, we observed that we obtain good results by computing $r^{(\pm)}(z)$ as Laurent-Pad\'{e} approximations for functions of the form $\omega^{(\pm)}(z)=\frac{\varphi^{(\pm)}(z)}{z-z_0^{(\pm)}}$, where $\varphi^{(\pm)}(z)$ are analytic in $\cc$, $z_0^{(\pm)} \in \cc$ and  $|z_0^{(\pm)}|>1$, i.e. $\omega^{(\pm)}(z)$ are analytic in the unit circle and meromorphic in the sets $\{z \in \mathbb{C}: |z|< \varrho^{(\pm)} \}$ for some $\varrho^{(\pm)}>1$  functions with one pole each.  From the Montessus de Ballore theorem regarding convergence of a Laurent-Pad\'{e} approximation (see, for example, \cite{B87}), we have that if $(N^{(\pm)}+1-M^{(\pm)}) \rightarrow \infty$, then $r^{(\pm)}(z) \rightarrow \frac{\varphi^{(\pm)}(z)}{z-z_0^{(\pm)}}$  uniformly on any compact subset of $\{z \in \mathbb{C}: |z|< \varrho^{(\pm)} \} \setminus \{z_0^{(\pm)}\}$ and $-\gamma_0^{(\pm)}/\gamma_1^{(\pm)} \rightarrow z_0^{(\pm)}$. The types of $r^{(\pm)}$ sufficiently depend on the number of singularities of $f$ and have an impact on the convergence of the neural network. Surprisingly, the locations of the   poles $z_0^{(\pm)}$ do not play an important role for the approximation results and are only used to introduce a rational structure of the neural network. 

\subsection{Computation of weights and biases of the hidden layers}
\label{ss3}

In this subsection, we present a new backpropagation-free method for the computation of the weights and biases of the hidden layers of $\Phi^{(\pm)}$. First, we  determine some additional parameters.
For given $N^{(\pm)}$ and $M^{(\pm)}$, we compute polynomials $p^{(+)}_{N^{(+)}}$ and $q^{(+)}_{M^{(+)}}$ from the condition 
\begin{equation}\label{pplus}
    q^{(+)}_{M^{(+)}}(z) f^{(+)}(z)- p^{(+)}_{N^{(+)}}(z) = \mathcal{O}(z^{N^{(+)}+M^{(+)}+1}),
\end{equation}
and polynomials $p^{(-)}_{N^{(-)}}$ and $q^{(-)}_{M^{(-)}}$ with respect to the argument $1/z$ from the condition
\begin{equation}\label{pminus}
    q^{(-)}_{M^{(-)}}(1/z) f^{(-)}(z)- p^{(-)}_{N^{(-)}}(1/z) = \mathcal{O}(z^{-(N^{(-)}+M^{(-)}+1)}).
\end{equation}
That means that coefficient vectors $\boldsymbol{q}^{(\pm)}:=(q_0^{(\pm)},\dots,q_{M^{(\pm)}}^{(\pm)})^T$ of the polynomials $q^{(+)}_{M^{(+)}}(z): = \sum_{i=0}^{M^{(+)}} q_i^{(+)} z^i$ and $q^{(-)}_{M^{(-)}}(1/z): = \sum_{i=0}^{M^{(-)}} q_i^{(-)} (1/z)^i$ can be computed from the  equations
$$
   \mathbf{C}_{N^{(\pm)},M^{(\pm)}}  \boldsymbol{q}^{(\pm)}=0,
   $$
where
$$
    \mathbf{C}_{N^{(\pm)},M^{(\pm)}}=\left( c_{\pm(N^{(\pm)}+k-\ell)}(f)  \right)_{k=1,\ell=0}^{M^{(\pm)}}   \in \cc^{M^{(\pm)}\times (M^{(\pm)}+1)}
    $$
are Toeplitz matrices.   Then the coefficient  vectors $\boldsymbol{p}^{(\pm)}:=(p_0^{(\pm)},\dots,p_{N^{(\pm)}}^{(\pm)})^T$ of polynomials $p^{(+)}_{N^{(+)}}(z):=\sum_{k=0}^{N^{(+)} } p_k^{(+)} z^k$ and $p^{(-)}_{N^{(-)}}(1/z):=\sum_{k=0}^{N^{(-)} } p_k^{(-)} (1/z)^k$ are determined by $  p_k^{(\pm)} =  \sum_{j=0}^{k} c_{\pm(k-j)}(f) q_j^{(\pm)}$ for $k=0,\dots,M^{(\pm)}$, and for $k=M^{(\pm)}+1,\dots,N^{(\pm)}$ via $  p_k^{(\pm)} =  \sum_{j=0}^{M^{(\pm)}} c_{\pm(k-j)}(f) q_j^{(\pm)}$.

As was already mentioned,  $M^{(\pm)}$ are the numbers of the neurons in the hidden layers of the components $\Phi^{(\pm)}$, and $\mathbf{W}_1^{(\pm)}:=(w_{11}^{(\pm)},\dots,w_{M^{(\pm)}1}^{(\pm)})^T \in \cc^{M^{(\pm)} \times 1}$ and  $\boldsymbol{b}_1^{(\pm)}:=(b_{1 1}^{(\pm)},\dots,b_{M^{(\pm)} 1}^{(\pm)})^T \in \cc^{M^{(\pm)}}$ are  weights and biases, respectively, of the hidden layers  of  $\Phi^{(\pm)}$. First, employing coefficients $q_i^{(\pm)}$, $i=0,\dots,M^{(\pm)}$, we determine parameters 
 $C_{k0}^{(\pm)}$ and $C_{k1}^{(\pm)}$, $k=1,\dots,M^{(\pm)}$, from the conditions
\begin{align}
    \prod\limits_{k=1}^{M^{(+)}} (C_{k0}^{(+)}+C_{k1}^{(+)} z ) & = \sum\limits_{i=0}^{M^{(+)}} q_i^{(+)} z^i,  \label{ccoef1} \\
    \prod\limits_{k=1}^{M^{(-)}} (C_{k0}^{(-)}+C_{k1}^{(-)} z^{-1} ) &= \sum\limits_{i=0}^{M^{(-)}} q_i^{(-)} z^{-i}. \label{ccoef2}
\end{align}
Then using also coefficients $\gamma_0^{(\pm)}$ and $\gamma_1^{(\pm)}$ of the denominators   of the activation functions (\ref{r1}), we determine weights and biases of $\Phi^{(\pm)}$ by
\begin{align}
    w_{k1}^{(\pm)} &=\frac{1}{\gamma_1^{(\pm)}} C^{(\pm)}_{k1}, \quad k=1,\dots,M^{(\pm)}, \label{w1}\\
    b_{k 1}^{(\pm)} & =\frac{1}{\gamma_1^{(\pm)}} (\gamma_0^{(\pm)} - C^{(\pm)}_{k0}),  \quad k=1,\dots,M^{(\pm)}.  \label{b1}
\end{align}
As will be shown in the convergence Theorem \ref{cont},  parameters $C_{k0}^{(\pm)}$ and $C_{k1}^{(\pm)}$, $k=1,\dots,M^{(\pm)}$ contain information about singularities of $f$.   Since formulas (\ref{ccoef1}) and (\ref{ccoef2}) give us $M^{(\pm)}-1$ degrees of freedom for parameters  $C_{k0}^{(\pm)}$ and $C_{k1}^{(\pm)}$, $k=1,\dots,M^{(\pm)}$, we compute $C_{k0}^{(\pm)}$, $k=1,\dots,M^{(\pm)}-1$ randomly.  Note, that from formulas (\ref{w1}) and (\ref{b1}) we can conclude that the random choice of $C_{k0}^{(\pm)}$, $k=1,\dots,M^{(\pm)}-1$ and the coefficients $\gamma_0^{(\pm)}$ and $\gamma_1^{(\pm)}$ of the activation functions have an impact on the values of weights and biases of the hidden layers of $\Phi^{(\pm)}$, but do not affect the locations of singularities of $f$, since those are controlled by parameters $q_i^{(\pm)}$.

\subsection{Computation of weights and biases of the output layers}
\label{ss4}

In this subsection, we describe a method to determine the weights $\mathbf{W}_2^{(\pm)}: = ( w_{1 2}^{(\pm)},\dots,w_{ M^{(\pm)} 2}^{(\pm)}) \in \cc^{1 \times M^{(\pm)} }$ and the biases $b_2^{(\pm)} \in \cc$ of the output layers of the neural network components $\Phi^{(\pm)}$. Employing coefficients $\alpha_j^{(\pm)}$, $j=0,\dots,N^{(\pm)}+1-M^{(\pm)}$ of the numerators of the  activation functions (\ref{r1}), and weights and biases of the hidden layers defined  by  (\ref{w1}) and (\ref{b1}), we compute parameters $d_{j k}^{(\ell,\pm)}$ by
\begin{equation}\label{alp}
d_{j k}^{(\ell,\pm)}:=\alpha_j^{(\pm)} \binom{j}{k} (- b_{\ell 1 }^{(\pm)} )^{j-k} (w_{\ell 1}^{(\pm)} )^{k}, \quad j=k,\dots,N^{(\pm)}+1-M^{(\pm)},
\end{equation}
and then $A_{k \ell}^{(\pm)}$ via
\begin{equation}\label{akl}
    A_{k \ell}^{(\pm)}:= \sum_{j=k}^{N^{(\pm)}+1-M^{(\pm)}} d_{j k}^{(\ell,\pm)}, \quad k=0,\dots,N^{(\pm)}+1-M^{(\pm)},
\end{equation}
for $\ell=1,\dots,M^{(\pm)}$. 
Finally, using coefficients $q_i^{(\pm)}$, $i=0,\dots,M^{(\pm)}$, $p_k^{(\pm)}$, $k=0,\dots,N^{(\pm)}$, parameters $C_{k0}^{(\pm)}$, $C_{k1}^{(\pm)}$, $k=1,\dots,M^{(\pm)}$ from (\ref{ccoef1})--(\ref{ccoef2}), and $ A_{k\ell}^{(\pm)}$, $k=0,\dots,N^{(\pm)}+1-M^{(\pm)}$, $\ell=1,\dots,M^{(\pm)}$ from (\ref{akl}), we determine the desired weights $w_{\ell 2}^{(+)}$, $\ell=1,\dots,M^{(+)}$ and the bias $b_2^{(+)}$ from the condition
\begin{equation}\label{w21}
\sum_{\ell=1}^{M^{(+)}} w_{\ell 2}^{(+)} F_{\ell}^{(+)}(z)  -b_2^{(+)} q^{(+)}_{M^{(+)}}(z)=p^{(+)}_{N^{(+)}}(z),
\end{equation}
and the  weights $w_{\ell 2}^{(-)}$, $\ell=1,\dots,M^{(-)}$ and the bias $b_2^{(-)}$ from the condition
\begin{equation}\label{w22}
\sum_{\ell=1}^{M^{(-)}} w_{\ell 2}^{(-)} F_{\ell}^{(-)}(1/z) -b_2^{(-)} q^{(-)}_{M^{(-)}}\left(1/z\right)=p^{(-)}_{N^{(-)}}\left(1/z\right),
\end{equation}
where  for $\ell=1,\ldots,M^{(\pm)}$,
\begin{align}
    F_{\ell}^{(+)}(z)&:= \sum_{k=0}^{N^{(+)}+1-M^{(+)}} A_{k\ell}^{(+)} z^k  \prod\limits_{\substack{i=1 \\ i\neq \ell}}^{M^{(+)}} (C_{i0}^{(+)}+C_{i1}^{(+)} z ),  \label{fp1} \\
    F_{\ell}^{(-)}(1/z)&:= \sum_{k=0}^{N^{(-)}+1-M^{(-)}} A_{k\ell}^{(-)} z^{-k}   \prod\limits_{\substack{i=1 \\ i\neq \ell}}^{M^{(-)}} (C_{i0}^{(-)}+C_{i1}^{(-)} z^{-1}  ). \label{fp2} 
\end{align}
Note that polynomials $F_{\ell}^{(\pm)}$  have degrees $N^{(\pm)}$. In Subsection \ref{sec43}, we consider a method for numerical computation of weights $\mathbf{W}_2^{(\pm)}$ and biases $b_2^{(\pm)}$.



\subsection{Convergence of the neural network}
\label{cnn}

In this subsection, we present the convergence result for the neural network $\Phi$  constructed in Subsections \ref{ss1}-\ref{ss4}.

\begin{theorem} \label{cont}
 Let  $f \in \mathcal{M}$ and $f$ has the Laurent expansion $f(z)=f^{(-)}(z)+f^{(+)}(z)$ as in (\ref{ls}) for each $z \in \mathcal{A}$.
Let further a neural network $\Phi:\cc\mapsto \cc$  be defined by $\Phi(z)=\Phi^{(+)}(z)+\Phi^{(-)}(z)$ as in (\ref{nn}). Components $\Phi^{(\pm)}$ are determined by rational activation functions of types $(N^{(\pm)}+1-M^{(\pm)},1)$  given by (\ref{r1}), by weights and biases of the hidden layers as in (\ref{w1}) and (\ref{b1}), respectively, and by  weights and biases of the output layers as in (\ref{w21}) and (\ref{w22}).
Then $\Phi$ satisfies the following properties:
\begin{itemize}
  \item[(1)] A neural network component $\Phi^{(+)}$ is a rational function of type $(N^{(+)}, M^{(+)})$, i.e. $\Phi^{(+)}(z)=\frac{p^{(+)}_{N^{(+)}}(z)}{q^{(+)}_{M^{(+)}}(z)}$.  Moreover, 
  $\lim\limits_{N^{(+)}\rightarrow \infty} \Phi^{(+)}(z) = f^{(+)}(z)$   uniformly on any compact subset of the set  $\left \{ z \in \mathbb{C}:  |z|< \Tilde{R}   \right \} \setminus \left \{ s_\ell^{(+)} \right \}_{\ell=1}^{\kappa^{(+)}}$ and  $q^{(+)}_{M^{(+)}} (z)\rightarrow   \prod_{\ell=1}^{\kappa^{(+)}} \left( 1- z/s_\ell^{(+)} \right)^{m_\ell^{(+)}}$ as $N^{(+)}\rightarrow \infty$.
    \item[(2)] A neural network component $\Phi^{(-)}$ is a rational function of type $(N^{(-)}, M^{(-)})$ of the  argument $1/z$, i.e. it can be represented as $\Phi^{(-)}(z)=\frac{p^{(-)}_{N^{(-)}}(z^{-1})}{q^{(-)}_{M^{(-)}}(z^{-1})}$. Furthermore,
  $\lim\limits_{N^{(-)}\rightarrow \infty} \Phi^{(-)}(z) = f^{(-)}(z)$  uniformly on any compact subset of the set  $\left \{ z \in \mathbb{C}:  |z|> \Tilde{r}   \right \} \setminus \left \{ s_\ell^{(-)} \right \}_{\ell=1}^{\kappa^{(-)}}$ and  $q^{(-)}_{M^{(-)}} (z^{-1}) \rightarrow   \prod_{\ell=1}^{\kappa^{(-)}} \left( 1- s_\ell^{(-)}/z \right)^{m_\ell^{(-)}}$ as $N^{(-)}\rightarrow \infty$.
\end{itemize}

\end{theorem}

\begin{proof}
First, we consider the neural network component $\Phi^{(+)}$ given by (\ref{f1}).  
Taking into account the  definition of the activation function  (\ref{r1}), weights and biases of the hidden layer (\ref{w1})--(\ref{b1}), as well as additional parameters (\ref{alp})--(\ref{akl}), we get for each $\ell=1,\dots,M^{(+)}$
\begin{align}
     r^{(+)}  & (w_{\ell 1}^{(+)}z-b_{\ell  1}^{(+)}) =\frac{\sum_{j=0}^{N^{(+)}+1-M^{(+)}} \alpha_j^{(+)} (w_{\ell 1}^{(+)}z-b_{\ell  1}^{(+)})^j}{ \gamma_0^{(+)}+ \gamma_1^{(+)}(w_{\ell 1}^{(+)}z-b_{\ell  1}^{(+)})} \notag \\
     & = \frac{\sum_{j=0}^{N^{(+)}+1-M^{(+)}} \sum_{k=0}^j d_{j k}^{(\ell, +)} z^k}{\gamma_0^{(+)}+ \gamma_1^{(+)} (\frac{1}{\gamma_1^{(+)}} C^{(+)}_{\ell 1} z - \frac{1}{\gamma_1^{(+)}} (\gamma_0^{(+)} - C^{(+)}_{\ell 0}))}= \frac{\sum_{k=0}^{N^{(+)}+1-M^{(+)}} A_{k \ell}^{(+)} z^k}{C_{\ell 0}^{(+)} + C_{\ell 1}^{(+)} z}.  \label{r11}
\end{align}
Thus, we can represent the column vector of neurons in the hidden layer $r^{(+)}(\mathbf{W}_1^{(+)} \, z - \boldsymbol{b}_1^{(+)})= \left( r^{(+)}(w_{\ell 1}^{(+)} \, z - b_{\ell 1}^{(+)}) \right)_{\ell=1}^{M^{(+)}} $ in the following form 
\begin{equation} \label{vnh}
    r^{(+)}(\mathbf{W}_1^{(+)} \, z - \boldsymbol{b}_1^{(+)})=  \left( \frac{\sum_{k=0}^{N^{(+)}+1-M^{(+)}}  A_{k \ell}^{(+)} z^k}{C_{\ell 0}^{(+)} + C_{\ell 1}^{(+)} z}  \right)_{\ell=1}^{M^{(+)}}.
\end{equation}
Using (\ref{vnh}), (\ref{ccoef1}) and (\ref{w21}), we rewrite  $\Phi^{(+)}$  as follows
\begin{align*}
    \Phi^{(+)}(z) 
    = & \frac{ w_{12}^{(+)} \sum_{k=0}^{N^{(+)}+1-M^{(+)}} A_{k1}^{(+)} z^k}{C_{10}^{(+)} + C_{11}^{(+)} z} + \ldots+\frac{ w_{M^{(+)} 2}^{(+)} \sum_{k=0}^{N^{(+)}+1-M^{(+)}} A_{k M^{(+)}}^{(+)} z^k}{C_{M^{(+)} 0}^{(+)} + C_{M^{(+)} 1}^{(+)} z}\\
    & -b_2^{(+)} =  \frac{\sum_{\ell=1}^{M^{(+)}} w_{\ell 2}^{(+)} F_{\ell}^{(+)}(z)  -b_2^{(+)} q^{(+)}_{M^{(+)}}(z)}{q^{(+)}_{M^{(+)}}(z)}= \frac{p^{(+)}_{N^{(+)}}(z)}{q^{(+)}_{M^{(+)}}(z)},
\end{align*}
where polynomials $p^{(+)}_{N^{(+)}}(z)$ and  $q^{(+)}_{M^{(+)}}(z)$ are determined from the condition (\ref{pplus}), and $F_{\ell}^{(+)}(z)$ are defined by (\ref{fp1}). Thus, we conclude that $\Phi^{(+)}$ is the Lauren-Pad\'{e} approximation of type  $(N^{(+)},M^{(+)})$ to   $f^{(+)}$ in the domain $\left \{ z \in \mathbb{C}:  |z|< \Tilde{R}   \right \}$, and the statement (1)  follows from the  Montessus de Ballore theorem regarding the convergence of a Laurent-Pad\'{e} approximation \cite{B87}.

Let us now consider the neural network component $\Phi^{(-)}(z)$ as in (\ref{f2}). Similarly as in (\ref{r11}), using definition of the activation function (\ref{r1}), and formulas  (\ref{w1})--(\ref{akl}), we obtain for each $\ell=1,\dots,M^{(-)}$
\begin{align}
     r^{(-)} &(w_{\ell 1}^{(-)}z^{-1}-b_{\ell  1}^{(-)}) =\frac{\sum_{j=0}^{N^{(-)}+1-M^{(-)}} \alpha_j^{(-)} (w_{\ell 1}^{(-)}z^{-1}-b_{\ell  1}^{(-)})^j}{ \gamma_0^{(-)}+ \gamma_1^{(-)}(w_{\ell 1}^{(-)}z^{-1}-b_{\ell  1}^{(-)})} \notag \\
     & = \frac{\sum_{j=0}^{N^{(-)}+1-M^{(-)}} \sum_{k=0}^j d_{j k}^{(\ell,-)} z^{-k}}{\gamma_0^{(-)}+ \gamma_1^{(-)} (\frac{1}{\gamma_1^{(-)}} C^{(-)}_{\ell 1} z^{-1} - \frac{1}{\gamma_1^{(-)}} (\gamma_0^{(-)} - C^{(-)}_{\ell 0}))}= \frac{\sum_{k=0}^{N^{(-)}+1-M^{(-)}} A_{k \ell}^{(-)} z^{-k}}{C_{\ell 0}^{(-)} + C_{\ell 1}^{(-)} z^{-1}}, \label{r12}
\end{align}
and for the column vector of the neurons in the hidden layer $r^{(-)}(\mathbf{W}_1^{(-)} \, z^{-1} - \boldsymbol{b}_{1}^{(-)})=\left( r^{(-)}(w_{\ell 1}^{(-)} \, z^{-1} - b_{\ell 1}^{(-)}) \right)_{\ell=1}^{M^{(-)}} $,
\begin{equation} \label{vnh2}
    r^{(-)}(\mathbf{W}_1^{(-)} \, z^{-1} - \boldsymbol{b}_{1}^{(-)})=  \left( \frac{\sum_{k=0}^{N^{(-)}+1-M^{(-)}}  A_{k \ell}^{(-)} z^{-k}}{C_{\ell 0}^{(-)} + C_{\ell 1}^{(-)} z^{-1}}  \right)_{\ell=1}^{M^{(-)}} .
\end{equation}
Applying  (\ref{vnh2}), (\ref{ccoef1}) and (\ref{w22}), we obtain for  $\Phi^{(-)}$  the following representation
\begin{align}
    \Phi^{(-)} (z) 
    = & \frac{ w_{12}^{(-)} \sum_{k=0}^{N^{(-)}+1-M^{(-)}} A_{k1}^{(-)} z^{-k}}{C_{10}^{(-)} + C_{11}^{(-)} z^{-1}} + \dots+\frac{ w_{M^{(-)} 2}^{(-)} \sum_{k=0}^{N^{(-)}+1-M^{(-)}} A_{k M^{(-)}}^{(-)} z^{-k}}{C_{M^{(-)}0}^{(-)} + C_{M^{(-)}1}^{(-)} z^{-1}}\notag \\
    & -b_2^{(-)}  =  \frac{\sum_{\ell=1}^{M^{(-)}} w_{\ell 2}^{(-)} F_{\ell}^{(-)}(1/z) -b_2^{(-)} q^{(-)}_{M^{(-)}}\left(1/z\right)}{q^{(-)}_{M^{(-)}}(1/z)}= \frac{p^{(-)}_{N^{(-)}}(1/z)}{q^{(-)}_{M^{(-)}}(1/z)}, \notag
\end{align}
where the polynomials $p^{(-)}_{N^{(-)}}(z^{-1})$ and $q^{(-)}_{M^{(-)}}(z^{-1})$ are computed from (\ref{pminus}) and $F_{\ell}^{(-)}(1/z)$ are defined by (\ref{fp2}). Therefore,  $ \Phi^{(-)}$ is the Laurent-Pad\'{e} approximation of type $(N^{(-)},M^{(-)})$ to $f^{(-)}$ in the domain $\left \{ z \in \mathbb{C}:  |z|> \Tilde{r}   \right \}$ and the statement (2)  follows again from the Montessus de Ballore theorem.
\end{proof}

\begin{remark}
 From the proof of Theorem \ref{cont}, namely from the formulas (\ref{vnh}) and (\ref{vnh2}), we observe that every neuron of the neural network is carrying information only about one singularity. That explains the structure of the components $\Phi^{(\pm)}$ with $M^{(\pm)}$ neurons in the hidden layers (see Figure \ref{nns}). From formulas (\ref{r11}) and (\ref{r12}), we also conclude that singularities $s_{\ell}^{(\pm)}$ of the function $f$ are computed by
 \begin{align}
     s_{\ell}^{(+)} &= (b^{(+)}_{\ell 1}+z_0^{(+)})/w^{(+)}_{\ell 1}, \quad \ell=1,\dots,M^{(+)}, \label{sin} \\
     s_{\ell}^{(-)} &= w^{(-)}_{\ell 1}/ (b^{(-)}_{\ell 1}+z_0^{(-)}), \quad \ell=1,\dots,M^{(-)}, \label{sin1}
\end{align}
where $z_0^{(\pm)}$ are poles of the activation functions $r^{(\pm)}$. By examining (\ref{sin}) and (\ref{sin1}), we arrive at our central idea, which is to use weights and biases of the hidden layers to scale and shift, respectively, the poles of the activation functions in order to find estimated locations of singularities (see Figure   \ref{figmain}). 
 Note that if a singularity $s_{\ell}^{(\pm)}$ has the multiplicity $m_{\ell}^{(\pm)}>1$,  then $m_{\ell}^{(\pm)}$ different neurons are carrying information about $s_{\ell}^{(\pm)}$. This question is discussed in the Numerical Experiments section.

\end{remark}

\begin{remark} From Theorem \ref{cont}, we can see that the neural network (\ref{nn}) has a structure suitable for detection of pole-type singularities. It is possible to modify the main steps of the construction of the neural network $\Phi$ described in Subsections \ref{ss1}--\ref{ss4}, such that we get a neural network for the detection of singularities of other types. The key component in the construction of $\Phi$ is using the  Laurent-Pad\'{e} method for rational approximation, which is known as a useful tool for estimating the locations of the pole-type singularities. As described in \cite{FH19}, the so-called quadratic Pad\'{e}  approximation contains square root singularities and can therefore provide additional information such as estimation of branch point locations of second order.  We can modify our method such that we obtain a neural network for the detection of pole-type singularities and branch points of second order. In that case, instead of the activation functions  (\ref{r1}) we have to use
$$
   r(z)=\frac{h(z)+\sqrt{d(z)}}{\gamma_1 z +\gamma_0 },
$$
where $d(z)$ is a polynomial of some degree $K$, and $K$ is the number of branch points of second order. Thus, in this case each neuron of the neural network will be carrying information about \textit{all branch points} of second order and the idea ``one singularity,  one neuron'' that holds for poles will not be kept. 

\end{remark}

\section{Robust algorithm for the construction of the neural network $\Phi$}
\label{rann}

In this subsection, we describe a robust numerical approach for constructing the neural network $\Phi$ as in (\ref{nn}), which is summarized in Algorithm \ref{alg3}.

\subsection{Numerical computation of Laurent coefficients}
\label{nclc}
To construct the neural network $\Phi$ for the approximation of a function $f \in \mathcal{M}$ given by (\ref{ls}), we employ a finite set of its Laurent coefficients $c_k(f)$ as input data. If they are not provided, we can compute these coefficients numerically using the procedure  described below.   
We choose a contour  $\Gamma$ as
\begin{equation}\label{count}
    \Gamma:=\Gamma^\rho= \{ z \in \cc : \, |z|=\rho, \, \rho>0 \}.
\end{equation}
Then, applying the change of variables $z=\rho \mathrm{e}^{\mathrm{i} \theta}$, $1\leq \theta \leq 2 \pi$, from (\ref{lc}) we obtain
\begin{align*}
c_k(f)=\frac{\rho^{-k} }{2\pi}  \int\limits_{0}^{2 \pi} f(\rho \mathrm{e}^{\mathrm{i} \theta}) \mathrm{e}^{-\mathrm{i} k \theta}  \, \mathrm{d} \theta= \rho^{-k} \hat{f}_k,
\end{align*}
where $\hat{f}_k$ are classical Fourier coefficients of the function $f(\rho \mathrm{e}^{\mathrm{i} \theta})$ and they can be approximated through a single Discrete Fourier Transform (DFT) applied to a given vector of function values $\boldsymbol{f}^{(2n)}:=\left(f(\rho \mathrm{e}^{\frac{2\pi \mathrm{i} j}{2n} })\right)_{j=0}^{2n-1}$ for some $n \in \nn$. Thus, we get 
\begin{equation}\label{apprlc}
c_k(f)\approx c_k^{(2n)}(f):=\frac{\rho^{-k}}{2n}
\begin{cases}
\hat{f}_{2n+k}^{(2n)}, & k=-n,\dots,-1, \\
\hat{f}_{k}^{(2n)}, & k=0,1\dots,n, \\
\end{cases}
\end{equation}
where $\hat{\boldsymbol{f}}^{(2n)}:=(\hat{f}_k^{(2n)})_{k=0}^{2n-1}$ is DFT of $\boldsymbol{f}^{(2n)}$ that can be computed by FFT with the complexity $\mathcal{O}(n \log n)$. We present these ideas in Algorithm \ref{alg2}.



\begin{algorithm}[ht]\caption{Numerical computation of Laurent coefficients}
\label{alg2}
\small{
\textbf{Input:}  $n\in \nn $, $\rho>0$,
 vector of function values $\boldsymbol{f}^{(2n)}=\left(f(\rho \mathrm{e}^{\frac{2\pi \mathrm{i} j}{2n} })\right)_{j=0}^{2n-1}$; 
\begin{enumerate}
\item Compute FFT  $\hat{\boldsymbol{f}}^{(2n)}:=(\hat{f}_k^{(2n)})_{k=0}^{2n-1}$ of a vector $\boldsymbol{f}^{(2n)}$.
\item Determine coefficients $c_k(f)$, $k=-n,\dots,n$ by (\ref{apprlc}).
\end{enumerate}

\noindent
\textbf{Output:}  Laurent coefficients $c_k(f)$, $k=-n,\dots,n$. \\
\textbf{Complexity:} $\mathcal{O}(n \log n)$.
}

\end{algorithm}

\begin{remark}
\textbf{1.} According to \cite{R18}, to choose the number $n \in \nn$, we can check the approximation error $
  |c_k(f) - c_k^{(2n)}(f)|\approx \frac{\rho^{-k}}{2 n}|\hat{f}_{k+2n}^{(4n)}|$, $k=-n,\dots,n$. 
In this case, we need to employ $2n$ additional function values, i.e. the vector $\boldsymbol{f}^{(4n)}:=\left(f(\rho \mathrm{e}^{\frac{2\pi \mathrm{i} j}{4n} })\right)_{j=0}^{4n-1}$, and  to compute $\hat{\boldsymbol{f}}^{(4n)}:=(\hat{f}_k^{(4n)})_{k=0}^{4n-1}$ as FFT of $\boldsymbol{f}^{(4n)}$.\\
\textbf{2. } To ensure stability of the algorithm, the parameter  $\rho>0$ has to be chosen as close as possible to $1$ \cite{R18}. Since the contour $\Gamma$ has to be located in the domain of analyticity of the function $f \in \mathcal{M}$, i.e. in the annulus $\mathcal{A}=\{z \in \mathbb{C}: r< |z|< R \}$, we also assume that $r<\rho<R$.  \\
\textbf{3. } We consider the case when a contour $\Gamma$ is given by (\ref{count}). As long as a contour $\Gamma$ satisfies the definition of Laurent coefficients, i.e.  $\Gamma$ is a regular and positively oriented simple loop around $0$ and $f$ is holomorphic on $\Gamma$, it can be also used in our method. We refer the interested reader to \cite{CC20} regarding numerical computation of Laurent coefficients for other choices of contours. 
\end{remark}


\subsection{Computation of the numbers of neurons in the hidden layers}

In Section \ref{spt}, for the construction of the neural network $\Phi$ we used the assumption that parameters $N^{(\pm)}$ and $M^{(\pm)}$  are known. Particularly, parameters  $M^{(\pm)}$ determine the number of singularities of $f$ in the  annuli $\Tilde{\mathcal{A}}^{(\pm)}$. Now we present Algorithm \ref{alg1} for the numerical computation of the parameters $N^{(\pm)}$ and $M^{(\pm)}$, which is  based on the iterative approach from \cite{GGT13}. We have given some upper bounds $N_1^{(\pm)}$ and $M_1^{(\pm)}$ for $N^{(\pm)}$ and $M^{(\pm)}$, respectively, and  vectors of Laurent coefficients $\boldsymbol{c}^{(\pm)}(f):=(c_0(f)/2,c_{\pm 1}(f),\dots,c_{\pm(N_1^{(\pm)}+M_1^{(\pm)})}(f))^T$. 




\begin{algorithm}[ht]\caption{Computation of parameters $N^{(\pm)}$  and $M^{(\pm)}$ and coefficient vectors  $\boldsymbol{p}^{(\pm)}=(p_0^{(\pm)},\dots,p_{N^{(\pm)}}^{(\pm)})^T$ and $\boldsymbol{q}^{(\pm)}=(q_0^{(\pm)},\dots,q_{M^{(\pm)}}^{(\pm)})^T$}
\label{alg1}
\small{
\textbf{Input:}  $N^{(\pm)}_1,  M^{(\pm)}_1\in \nn $ (upper bounds for $N^{(\pm)}$  and $M^{(\pm)}$), $\mathrm{tol}=10^{-14}$;\\
\phantom{\textbf{Input:}}  vectors of Laurent coefficients $\boldsymbol{c}^{(\pm)}(f):=(c_0(f)/2,c_{\pm 1}(f),\dots,c_{\pm(N_1^{(\pm)}+M_1^{(\pm)})}(f))^T$; 
\begin{enumerate}
\item Compute $\tau^{(\pm)}=\mathrm{tol}\cdot \| \boldsymbol{c}^{(\pm)}(f)\|_2$.
\item Construct  Toeplitz matrices $\mathbf{C}_{N_1^{(\pm)},M_1^{(\pm)}}$ as in (\ref{t1}) and compute their SVDs  (\ref{svd1}). Determine  numbers $\mu^{(\pm)} \leq M_1^{(\pm)}$  of singular values of $\mathbf{C}_{N_1^{(\pm)},M_1^{(\pm)}}$ that are larger than $\tau^{(\pm)}$.
\item If $\mu^{(\pm)}< M_1^{(\pm)}$, reduce $M_1^{(\pm)}$ to $\mu^{(\pm)}$ and $N_1^{(\pm)}$ to $N_1^{(\pm)}-(M_1^{(\pm)}-\mu^{(\pm)})$ and go back to  Step 2.
\item Compute   vectors $\boldsymbol{q}^{(\pm)}=(q_0^{(\pm)},\dots,q_{M_1^{(\pm)}}^{(\pm)})^T$ from the equation (\ref{eqq}) and  vectors $\boldsymbol{p}^{(\pm)}=(p_0^{(\pm)},\dots,p_{N_1^{(\pm)}}^{(\pm)})^T$ via (\ref{coefhp1}) and (\ref{coefhp2}). 
\item If for some integer $\lambda^{(\pm)}\geq 1$,  $|q_0^{(\pm)}|,\dots,|q_{\lambda^{(\pm)}-1}^{(\pm)}|\leq \mathrm{tol}$, redefine $q_j^{(\pm)}:=q_{j+\lambda^{(\pm)}}^{(\pm)}$ for $j=0,\dots,M_1^{(\pm)}-\lambda^{(\pm)}$ and $p_j^{(\pm)}:=p_{j+\lambda^{(\pm)}}^{(\pm)}$ for $j=0,\dots,N_1^{(\pm)}-\lambda^{(\pm)}$ and reduce  $M_1^{(\pm)}$ to $M_1^{(\pm)}-\lambda^{(\pm)}$ and  $N_1^{(\pm)} $ to $N_1^{(\pm)}-\lambda^{(\pm)}$.
\item  If for some integer $\lambda^{(\pm)}\geq 1$, $|q_{M_1^{(\pm)}+1-\lambda^{(\pm)}}^{(\pm)}|,\dots,|q_{M_1^{(\pm)}}^{(\pm)}|\leq \mathrm{tol}$, set $M^{(\pm)}:=M_1^{(^{(\pm)})}-\lambda^{(\pm)}$ and determine  $\boldsymbol{q}^{(\pm)}=(q_0^{(\pm)},\dots,q_{M^{(\pm)}}^{(\pm)})^T$.
\item  If for some integer $\lambda^{(\pm)}\geq 1$, $|p_{N_1^{(^{(\pm)})}+1-\lambda^{(\pm)}}^{(\pm)}|,\dots,|p_{N_1^{(\pm)}}^{(\pm)}|\leq \mathrm{tol}$, set  $N^{(\pm)}:=N_1^{(\pm)}-\lambda^{(\pm)}$ and determine $\boldsymbol{p}^{(\pm)}=(p_0^{(\pm)},\dots,p_{N^{(\pm)}}^{(\pm)})^T$.
\end{enumerate}

\noindent
\textbf{Output:} $N^{(\pm)}, M^{(\pm)} \in \nn$, vectors  $\boldsymbol{p}^{(\pm)}=(p_0^{(\pm)},\dots,p_{N^{(\pm)}}^{(\pm)})^T$ and $\boldsymbol{q}^{(\pm)}=(q_0^{(\pm)},\dots,q_{M^{(\pm)}}^{(\pm)})^T$.  \\
\textbf{Complexity:} $\mathcal{O}((M_1^{(\pm)})^3 \ (2+\log(\delta^{(\pm)}+1)))$, where $\delta^{(\pm)}=\min\{M_1^{(\pm)}-M^{(\pm)}, N_1^{(\pm)}-N^{(\pm)} \}$.
}

\end{algorithm}

For given $\mathrm{tol}>0$, we determine $\tau^{(\pm)}=\mathrm{tol}\cdot \| \boldsymbol{c}^{(\pm)}(f)\|_2$ and construct
  Toeplitz matrices
\begin{equation}\label{t1}
    \mathbf{C}_{N_1^{(\pm)},M_1^{(\pm)}}=\left( c_{\pm(N_1^{(\pm)}+k-\ell)}(f)  \right)_{k=1,\ell=0}^{M_1^{(\pm)}}   \in \cc^{M_1^{(\pm)}\times (M_1^{(\pm)}+1)}.
\end{equation}
 First, we compute SVDs of $ \mathbf{C}_{N_1^{(\pm)},M_1^{(\pm)}}$,
\begin{equation}\label{svd1}
   \mathbf{C}_{N_1^{(\pm)},M_1^{(\pm)}}=\mathbf{U}^{(\pm)} \mathbf{\Sigma}^{(\pm)} (\mathbf{V}^{(\pm)})^{*},
\end{equation}
where $\mathbf{U}^{(\pm)}  \in \cc^{M_1^{(\pm)}\times M_1^{(\pm)}}$ and $\mathbf{V}^{(\pm)}  \in \cc^{(M_1^{(\pm)}+1)\times (M_1^{(\pm)}+1)}$ are unitary and $\mathbf{\Sigma}^{(\pm)}  \in \rr^{M_1^{(\pm)}\times (M_1^{(\pm)}+1)}$ are real diagonal matrices with diagonal entries $\sigma_1^{(\pm)}\geq \sigma_2^{(\pm)} \geq \dots \geq \sigma_{M_1^{(\pm)}}^{(\pm)}\geq 0$, which are singular values of $\mathbf{C}_{N_1^{(\pm)},M_1^{(\pm)}}$. Let $\mu^{(\pm)}\leq M_1^{(\pm)}$ be the numbers of singular values of $\mathbf{C}_{N_1^{(\pm)},M_1^{(\pm)}}$ that are larger than $\tau^{(\pm)}$. If $\mu^{(\pm)}< M_1^{(\pm)}$, we reduce $M_1^{(\pm)}$ to $\mu^{(\pm)}$ and $N_1^{(\pm)}$ to $N_1^{(\pm)}-(M_1^{(\pm)}-\mu^{(\pm)})$ and go back to the same step of computing SVDs of the matrices (\ref{svd1}). In \cite{GGT13}, it was shown that this step is performed no more than $2+\log(\delta^{(\pm)}+1)$ times, where $\delta^{(\pm)}=\min\{M_1^{(\pm)}-M^{(\pm)}, N_1^{(\pm)}-N^{(\pm)} \}$. Computation of SVDs (\ref{svd1}) determines the overall computational cost of Algorithm \ref{alg1} which is $\mathcal{O}((M_1^{(\pm)})^3 \ (2+\log(\delta^{(\pm)}+1)))$ flops.

Further, we  compute  vectors $\boldsymbol{q}^{(\pm)}=(q_0^{(\pm)},\dots,q_{M_1^{(\pm)}}^{(\pm)})^T$ from the equations
\begin{equation}\label{eqq}
   \mathbf{C}_{N_1^{(\pm)},M_1^{(\pm)}}  \boldsymbol{q}^{(\pm)}=0,
\end{equation}
i.e. $\boldsymbol{q}^{(\pm)}$ are the null right singular vectors of the matrices $\mathbf{C}_{N_1^{(\pm)},M_1^{(\pm)}}$. Then vectors $\boldsymbol{p}^{(\pm)}=(p_0^{(\pm)},\dots,p_{N_1^{(\pm)}}^{(\pm)})^T$ are determined by 
\begin{align}
    p_k^{(\pm)} &=  \sum\limits_{j=0}^{k} c_{\pm(k-j)}(f) q_j^{(\pm)}, \quad  k=0,\dots,M_1^{(\pm)}, \label{coefhp1} \\
    p_k^{(\pm)} &=  \sum\limits_{j=0}^{M_1^{(\pm)}} c_{\pm(k-j)}(f) q_j^{(\pm)}, \quad  k=M_1^{(\pm)}+1,\dots,N_1^{(\pm)}. \label{coefhp2}
\end{align}
Next, we check if $|q_0^{(\pm)}|,\dots,|q_{\lambda^{(\pm)}-1}^{(\pm)}|\leq \mathrm{tol}$ for some integer $\lambda^{(\pm)}\geq 1$ and zero the first $\lambda^{(\pm)}$ coefficients of $\boldsymbol{q}^{(\pm)}$  and $\boldsymbol{p}^{(\pm)}$. Then we  redefine $q_j^{(\pm)}:=q_{j+\lambda^{(\pm)}}^{(\pm)}$ for $j=0,\dots,M_1^{(\pm)}-\lambda^{(\pm)}$ and $p_j^{(\pm)}:=p_{j+\lambda^{(\pm)}}^{(\pm)}$ for $j=0,\dots,M_1^{(\pm)}-\lambda^{(\pm)}$ and reduce  $M_1^{(\pm)}$ to $M_1^{(\pm)}-\lambda^{(\pm)}<M_1^{(\pm)}$ and $N_1^{(\pm)}$ to $N_1^{(\pm)}-\lambda^{(\pm)}< N_1^{(\pm)}$.

Finally,  if $|q_{M_1^{(\pm)}+1-\lambda^{(\pm)}}^{(\pm)}|,\dots,|q_{M_1^{(\pm)}}^{(\pm)}|\leq \mathrm{tol}$ for some integer $\lambda^{(\pm)}\geq 1$, we remove the last $\lambda^{(\pm)}$ coefficients of $\boldsymbol{q}^{(\pm)}$, reduce $M_1^{(\pm)}$ to $M^{(\pm)}:=M_1^{(^{(\pm)})}-\lambda^{(\pm)}$ and get the final vectors $\boldsymbol{q}^{(\pm)}=(q_0^{(\pm)},\dots,q_{M^{(\pm)}}^{(\pm)})^T$. Similarly,  if $|p_{N_1^{(^{(\pm)})}+1-\lambda^{(\pm)}}^{(\pm)}|,\dots,|p_{N_1^{(\pm)}}^{(\pm)}|\leq \mathrm{tol}$ for some integer $\lambda^{(\pm)}\geq 1$, we remove the last $\lambda^{(\pm)}$ coefficients of $\boldsymbol{p}^{(\pm)}$, reduce $N_1^{(\pm)}$ to $N^{(\pm)}:=N_1^{(\pm)}-\lambda^{(\pm)}$ and obtain the final vectors $\boldsymbol{p}^{(\pm)}=(p_0^{(\pm)},\dots,p_{N^{(\pm)}}^{(\pm)})^T$.

\begin{algorithm}[ht]\caption{Computation of the coefficient vectors $\boldsymbol{\gamma}^{(\pm)}=(\gamma_0^{(\pm)},\gamma_1^{(\pm)})^T$ and $\boldsymbol{\alpha}^{(\pm)}=(\alpha_0^{(\pm)},\dots,\alpha_{N^{(\pm)}+1-M^{(\pm)}}^{(\pm)})^T$ of the activation functions $r^{(\pm)}$ as in (\ref{r1}) }
\label{alg5}
\small{
\textbf{Input:}  $N^{(\pm)},  M^{(\pm)}\in \nn $, $n\in \nn$;\\
\phantom{\textbf{Input:}} functions $\omega^{(\pm)}(z)=\frac{\varphi^{(\pm)}(z)}{z-z_0^{(\pm)}}$,  where $\varphi^{(\pm)}(z)$ are analytic in $\cc$, $z_0^{(\pm)} \in \cc$ and  $|z_0^{(\pm)}|>1$;
\begin{enumerate}
\item Generate values $\left(\boldsymbol{\omega^{(\pm)}}\right)^{(2n)}=\left(\omega^{(\pm)}( \mathrm{e}^{\frac{2\pi \mathrm{i} j}{2n} })\right)_{j=0}^{2n-1}$, compute  Laurent coefficients $c_k(\omega^{(\pm)})$, $k=-n,\dots,n$ by Algorithm \ref{alg2}, and set  $\boldsymbol{c}^{(\pm)} (\omega^{(\pm)}):=(c_0(\omega^{(\pm)})/2,c_{\pm 1}(\omega^{(\pm)}),\dots,c_{\pm(N^{(\pm)}+2-M^{(\pm)})}(\omega^{(\pm)}))^T$.
\item Create Toeplitz matrices $ \mathbf{C}_{N^{(\pm)}+1-M^{(\pm)},1}$ as in (\ref{t1}) with Laurent coefficients of $\omega^{(\pm)}$ and compute vectors $\boldsymbol{\gamma}^{(\pm)}=(\gamma_0^{(\pm)},\gamma_1^{(\pm)})^T$ from the conditions $\mathbf{C}_{N^{(\pm)}+1-M^{(\pm)},1}  \boldsymbol{\gamma}^{(\pm)}=0$.
\item Compute vectors $\boldsymbol{\alpha}^{(\pm)}=(\alpha_0^{(\pm)},\dots,\alpha_{N^{(\pm)}+1-M^{(\pm)}}^{(\pm)})^T$ by
$
\alpha_k^{(\pm)}=  \sum_{j=0}^{k} c_{\pm(k-j)}(\omega^{(\pm)}) \gamma_j^{(\pm)}, \quad  k=0,1,
$
and
$
\alpha_k^{(\pm)}=  \sum_{j=0}^{1} c_{\pm(k-j)}(\omega^{(\pm)}) \gamma_j^{(\pm)}, \quad  k=2,\dots,N^{(\pm)}+1-M^{(\pm)}.
$
\end{enumerate}

\noindent
\textbf{Output:} Coefficient vectors  $\boldsymbol{\gamma}^{(\pm)}=(\gamma_0^{(\pm)},\gamma_1^{(\pm)})^T$ and $\boldsymbol{\alpha}^{(\pm)}=(\alpha_0^{(\pm)},\dots,\alpha_{N^{(\pm)}+1-M^{(\pm)}}^{(\pm)})^T$ . \\
\textbf{Complexity:} $\mathcal{O}(n \log n)$.}
\end{algorithm}

\begin{remark}
\textbf{1.}  We consider functions $f \in \mathcal{M}$ with the conditions $M^{(+)}>0$ and/or $M^{(-)}>0$ for total multiplicities $M^{(\pm)}$ of pole-type singularities of $f$ in the annuli $\Tilde{\mathcal{A}}^{(\pm)}$. These parameters $M^{(\pm)}$  are computed numerically by Algorithm \ref{alg1}. If one of them is equal to zero, i.e. $M^{(+)}=0$ or $M^{(-)}=0$, then the neural network $\Phi$ contains only one component, $\Phi^{(-)}$ or $\Phi^{(+)}$, respectively. \\
\textbf{2.} As recommended in \cite{GGT13}, we employ $\mathrm{tol}=10^{-14}$ for numerical experiments.
\end{remark}

\subsection{Construction of the adaptive activation functions}

In this subsection, we present Algorithm \ref{alg5} for the construction of rational activation functions (\ref{r1}). As mentioned in Subsection \ref{ss1}, the activation functions $r^{(\pm)}$ are constructed as the Laurent-Pad\'{e} approximations to $\omega^{(\pm)}(z)=\frac{\varphi^{(\pm)}(z)}{z-z_0^{(\pm)}}$,  where $\varphi^{(\pm)}(z)$ are analytic in $\cc$, $z_0^{(\pm)} \in \cc$ and  $|z_0^{(\pm)}|>1$. We use a method from \cite{GGT13} with the assumption that the type of rational approximant is known. The complexity of this method is $\mathcal{O}(n \log n)$, dominated by the complexity of the numerical computation of Laurent coefficients of $\omega^{(\pm)}$ via Algorithm \ref{alg2}.


\begin{algorithm}[ht]\caption{Computation of weights and biases of $\Phi^{(\pm)}$ }
\label{alg4}
\small{
\textbf{Input:} $N^{(\pm)}, M^{(\pm)} \in \nn$, vectors  $\boldsymbol{p}^{(\pm)}=(p_0^{(\pm)},\dots,p_{N^{(\pm)}}^{(\pm)})^T$ and $\boldsymbol{q}^{(\pm)}=(q_0^{(\pm)},\dots,q_{M^{(\pm)}}^{(\pm)})^T$;\\
\phantom{\textbf{Input:}}  vectors $\boldsymbol{\gamma}^{(\pm)}=(\gamma_0^{(\pm)},\gamma_1^{(\pm)})^T$ and $\boldsymbol{\alpha}^{(\pm)}=(\alpha_0^{(\pm)},\dots,\alpha_{N^{(\pm)}+1-M^{(\pm)}}^{(\pm)})^T$; \\ 
\phantom{\textbf{Input:}} $a,b,c,d \in \rr$, $a<b$, $c<d$, $n \in \nn$;
\begin{enumerate}
\item Compute  $C_{k 0}^{(\pm)}$, $k=1,\dots,M^{(\pm)}-1$, as random numbers from  $[a,b]+\mathrm{i} [c,d]$, and $C_{M^{(\pm)}0}^{(\pm)}$ by (\ref{cm}). Determine $C_{k1}^{(\pm)}$, $k=1,\ldots, M^{(\pm)}$  from  (\ref{ccoef1})--(\ref{ccoef2}).
\item Compute
weights $\mathbf{W}_1^{(\pm)}=(w_{11}^{(\pm)},\dots,w_{M^{(\pm)}1}^{(\pm)})^T $ and biases $\boldsymbol{b}_1^{(\pm)}=(b_{1 1}^{(\pm)},\dots,b_{M^{(\pm)} 1}^{(\pm)})^T $ of the hidden layers of  $\Phi^{(\pm)}$ via (\ref{w1}) and (\ref{b1}).
\item Compute parameters $d_{j k}^{(\ell,\pm)}$, $j=k,\dots,N^{(\pm)}+1-M^{(\pm)}$ via (\ref{alp}) and $A_{k \ell}^{(\pm)}$  $k=0,\dots,N^{(\pm)}+1-M^{(\pm)}$, $\ell=1,\dots,M^{(\pm)}$ via (\ref{akl}). 
\item Construct matrices $\mathbf{F}^{(\pm)}$ and vectors $\boldsymbol{Q}^{(\pm)}$ and $\boldsymbol{P}^{(\pm)}$ as in (\ref{fqp}) via their entries (\ref{fm}), and compute   weights $\mathbf{W}_2^{(\pm)} = ( w_{1 2}^{(\pm)},\dots,w_{ M^{(\pm)} 2}^{(\pm)}) $ and  biases $b_2^{(\pm)} $ of the output layers of $\Phi^{(\pm)}$ as the least squares solutions of the linear systems  (\ref{lsls}).
\end{enumerate}

\noindent
\textbf{Output:} Weights  $\mathbf{W}_1^{(\pm)}=(w_{11}^{(\pm)},\dots,w_{M^{(\pm)}1}^{(\pm)})^T \in \cc^{M^{(\pm)} \times 1}$ and biases $\boldsymbol{b}_1^{(\pm)}=(b_{1 1}^{(\pm)},\dots,b_{M^{(\pm)} 1}^{(\pm)})^T \in \cc^{M^{(\pm)} }$ of the hidden layers and weights $\mathbf{W}_2^{(\pm)} = ( w_{1 2}^{(\pm)},\dots,w_{ M^{(\pm)} 2}^{(\pm)}) \in \cc^{1 \times M^{(\pm)} }$ and biases $b_2^{(\pm)} \in \cc$ of the output layers of the neural network components $\Phi^{(\pm)} $. \\
\textbf{Complexity:} $\mathcal{O}(n(M^{(\pm)})^2)$.}
\end{algorithm}

\subsection{Computation  of weights and biases of the neural network components}
\label{sec43}
From formulas (\ref{w1}) and (\ref{b1}), we see that 
 weights and biases of the hidden layers depend on the coefficients $\gamma_0^{(\pm)}$ and  $\gamma_1^{(\pm)}$ of the activation functions,  and on the  parameters  $C_{k0}^{(\pm)}$ and $C_{k1}^{(\pm)}$, $k=1,\dots,M^{(\pm)}$ as in (\ref{ccoef1}) and (\ref{ccoef2}).  As was already discussed in Subsection \ref{ss3}, to compute $C_{k0}^{(\pm)}$ and $C_{k1}^{(\pm)}$, $k=1,\dots,M^{(\pm)}$ we can use $M^{(\pm)}-1$ degrees of freedom. Thus, we apply the following idea:  $C_{k0}^{(\pm)}$, $k=1,\dots,M^{(\pm)}-1$ are  computed as random numbers from  $[a,b]+\mathrm{i} [c,d]$ for some $a,b,c,d \in \rr$ ($a<b$, $c<d$) and $C_{M^{(\pm)}0}^{(\pm)}$ are determined via
\begin{equation}\label{cm}
C_{M^{(\pm)}0}^{(\pm)}=\frac{q_0^{(\pm)}}{C_{1 0}^{(\pm)}\cdot \ldots \cdot C^{(\pm)}_{(M^{(\pm)}-1) 0} }.
\end{equation}
Then $C_{k1}^{(\pm)}$, $k=1,\ldots, M^{(\pm)}$ are computed from the conditions (\ref{ccoef1}) and (\ref{ccoef2}).


Finally, we consider a method for numerical computations of weights $\mathbf{W}_2^{(\pm)}$ and biases $ b_2^{(\pm)} $ of the output layers of $\Phi^{(\pm)}$.  Let us have $2n$ points $z_m=\mathrm{e}^{\frac{2\pi m}{2n} \mathrm{i}}$,  $m=0,\dots,2n-1$. According to (\ref{w21}) and (\ref{w22}),  vectors $(\mathbf{W}_2^{(\pm)}  \,  | \, b_2^{(\pm)})^{T}  \in \cc^{ (M^{(\pm)}+1) \times 1 }$ can be computed as least squares solutions of the linear systems
\begin{equation}\label{lsls}
(\mathbf{F}^{(\pm)} \,  | \, \boldsymbol{Q}^{(\pm)} )  \,  (\mathbf{W}_2^{(\pm)}  \,  | \, b_2^{(\pm)})^{T} = \boldsymbol{P}^{(\pm)} ,
\end{equation}
where the corresponding matrices $\mathbf{F}^{(\pm)}\in \cc^{ 2n \times M^{(\pm)}  }$ and vectors $\boldsymbol{Q}^{(\pm)}, \boldsymbol{P}^{(\pm)} \in \cc^{ 2n }$,
\begin{equation}\label{fqp}
\mathbf{F}^{(\pm)} = \left(F^{(\pm)}_{m \ell} \right)_{m=0,\ell=1}^{2n-1,M^{(\pm)}},  \quad \boldsymbol{Q}^{(\pm)}  = ( Q_m^{(\pm)} )_{m=0}^{2n-1} ,  \quad \boldsymbol{P}^{(\pm)}  = ( P_m^{(\pm)} )_{m=0}^{2n-1} 
\end{equation}
are defined by their entries
\begin{equation}\label{fm}
F^{(\pm)}_{m \ell} = F^{(\pm)}_{\ell} (z_m), \quad  Q_m^{(\pm)} =-q^{(\pm)}_{M^{(\pm)}}(z_m), \quad  P_m^{(\pm)} = p^{(\pm)}_{N^{(\pm)}} (z_m).
\end{equation}
We present the main ideas of the computation of weights and biases of the components $\Phi^{(\pm)}$ in Algorithm \ref{alg4}. The numerical cost of this algorithm is determined by the complexity of the least squares solutions of the systems of linear equations (\ref{lsls}) which is  $\mathcal{O}(n(M^{(\pm)})^2)$. Note also that we use the same number $n \in \nn$ for Algorithms \ref{alg2}, \ref{alg5} and \ref{alg4}.

\begin{algorithm}[ht]\caption{Construction of the neural network components $\Phi^{(\pm)}$ }
\label{alg3}
\small{
\textbf{Input:}  $N^{(\pm)}_1,  M^{(\pm)}_1\in \nn $, $\rho>0$, $n\in \nn$, $a,b, c,d \in \rr$, $a<b$, $c<d$,  $\mathrm{tol}=10^{-14}$; \\ 
\phantom{\textbf{Input:}} functions $\omega^{(\pm)}(z)=\frac{\varphi^{(\pm)}(z)}{z-z_0^{(\pm)}}$,  where $\varphi^{(\pm)}(z)$ are analytic in $\cc$, $z_0^{(\pm)} \in \cc$ and  $|z_0^{(\pm)}|>1$; \\
\phantom{\textbf{Input:}} vector of function values $\boldsymbol{f}^{(2n)}=\left(f(\rho \mathrm{e}^{\frac{2\pi \mathrm{i} j}{2n} })\right)_{j=0}^{2n-1}$; 
\begin{enumerate}
\item Compute  Laurent coefficients $c_k(f)$, $k=-n,\dots,n$ by Algorithm \ref{alg2} and set  $\boldsymbol{c}^{(\pm)}(f):=(c_0(f)/2,c_{\pm 1}(f),\dots,c_{\pm(N_1^{(\pm)}+M_1^{(\pm)})}(f))^T$.
\item Determine parameters $N^{(\pm)}$  and $M^{(\pm)}$, and  vectors  $\boldsymbol{p}^{(\pm)}=(p_0^{(\pm)},\dots,p_{N^{(\pm)}}^{(\pm)})^T$ and $\boldsymbol{q}^{(\pm)}=(q_0^{(\pm)},\dots,q_{M^{(\pm)}}^{(\pm)})^T$ by Algorithm \ref{alg1}.
\item Compute coefficient vectors  $\boldsymbol{\alpha}^{(\pm)}=(\alpha_0^{(\pm)},\dots,\alpha_{N^{(\pm)}+1-M^{(\pm)}}^{(\pm)})^T$ and $\boldsymbol{\gamma}^{(\pm)}=(\gamma_0^{(\pm)},\gamma_1^{(\pm)})^T$ of the   activation functions $r^{(\pm)}$ by Algorithm \ref{alg5}.
\item Compute weights  $\mathbf{W}_1^{(\pm)}=(w_{11}^{(\pm)},\dots,w_{M^{(\pm)}1}^{(\pm)})^T$ and biases $\boldsymbol{b}_1^{(\pm)}=(b_{1 1}^{(\pm)},\dots,b_{M^{(\pm)} 1}^{(\pm)})^T$ of the hidden layers and weights $\mathbf{W}_2^{(\pm)} = ( w_{1 2}^{(\pm)},\dots,w_{ M^{(\pm)} 2}^{(\pm)})$ and  biases $b_2^{(\pm)} $ of the output layers of the neural network components $\Phi^{(\pm)} $ by Algorithm \ref{alg4}. Compute poles of the function $f$ via (\ref{sin}) and (\ref{sin1}).
\end{enumerate}

\noindent
\textbf{Output:} Parameters   $N^{(\pm)},  M^{(\pm)}\in \nn $,  coefficient vectors $\boldsymbol{\alpha}^{(\pm)}=(\alpha_0^{(\pm)},\dots,\alpha_{N^{(\pm)}+1-M^{(\pm)}}^{(\pm)})$ and $\boldsymbol{\gamma}^{(\pm)}=(\gamma_0^{(\pm)},\gamma_1^{(\pm)})$  of the activation functions,
weights $\mathbf{W}_1^{(\pm)}=(w_{11}^{(\pm)},\dots,w_{M^{(\pm)}1}^{(\pm)})^T \in \cc^{M^{(\pm)} \times 1}$ and biases $\boldsymbol{b}_1^{(\pm)}=(b_{1 1}^{(\pm)},\dots,b_{M^{(\pm)} 1}^{(\pm)})^T \in \cc^{M^{(\pm)} }$ of the hidden layers,
weights  $\mathbf{W}_2^{(\pm)} = ( w_{1 2}^{(\pm)},\dots,w_{ M^{(\pm)} 2}^{(\pm)}) \in \cc^{1 \times M^{(\pm)} }$ and biases $b_2^{(\pm)} \in \cc$ of the output layers of  $\Phi^{(\pm)}$. \\
\textbf{Complexity:} $\mathcal{O}(n \log n+(M_1^{(\pm)})^3 \ (2+\log(\delta^{(\pm)}+1))+n(M^{(\pm)})^2 )$, where $\delta^{(\pm)}=\min\{M_1^{(\pm)}-M^{(\pm)}, N_1^{(\pm)}-N^{(\pm)} \}$.
}

\end{algorithm}

\subsection{Robust numerical algorithm for the construction of the neural network $\Phi$}

Finally, we summarize our ideas from Subsections \ref{nclc}-\ref{sec43} in Algorithm \ref{alg3} for construction of the neural network $\Phi$ as in (\ref{nn}) for approximation of functions $f \in \mathcal{M}$.  
The numerical cost of this algorithm is $\mathcal{O}(n \log n+(M_1^{(\pm)})^3 \ (2+\log(\delta^{(\pm)}+1))+n(M^{(\pm)})^2 )$, where $\delta^{(\pm)}=\min\{M_1^{(\pm)}-M^{(\pm)}, N_1^{(\pm)}-N^{(\pm)} \}$. We remind the reader that as inputs  we employ a finite vector of Laurent coefficients $c_k(f)$ of $f$. In Algorithm \ref{alg3}, we compute  $c_k(f)$, $k=-n,\ldots,n$ via Algorithm \ref{alg2} using values of $f$ at $2 n$ equidistant points from the contour  $\Gamma$, and then for the method we only need  $c_k(f)$ for $k=-(N_1^{(\pm)}+M_1^{(\pm)}), \ldots ,N_1^{(\pm)}+M_1^{(\pm)}$. Thus, we assume that the number $n$ is large enough (such that the condition $n\geq \max\{ N_1^{(+)}+M_1^{(+)}, N_1^{(-)}+M_1^{(-)} \}$ holds). 
Note that the computational cost of the method remains unchanged if Laurent coefficients are already provided.





\section{Numerical experiments}
\label{nex}

We now discuss  numerical experiments to illustrate the performance of our method, and then show how it can be used to construct extensions of the time-dependent solutions  of nonlinear PDEs. The algorithm is implemented in \textsc{Matlab} and use IEEE standard floating point arithmetic with double precision.

\subsection{Learning functions with singularities}
\label{ex1}

First, we consider the example of approximation of a function $f \in \mathcal{M}$ defined by
\begin{equation}\label{exfun}
    f(z)=\frac{\cos(z)}{\prod_{j=1}^5(z-\zeta_j)^2 \prod_{j=1}^{20}(z-\varsigma_j) \prod_{j=1}^{10}(z-\pi_j)},
\end{equation}
where  $\zeta_j=0.8 \, \mathrm{e}^{\frac{\pi \mathrm{i} j}{5}}$, $j=1,\dots,5$ are singularities with double multiplicities,  $\varsigma_j=0.7 \, \mathrm{e}^{\frac{\pi \mathrm{i} j}{10}}$, $j=1,\dots,10$, $\varsigma_j=-0.9 \, \mathrm{e}^{\frac{\pi \mathrm{i} j}{10}}$, $j=11,\dots,20$, and $\pi_j=1.2\, \mathrm{e}^{\frac{\pi \mathrm{i} (j+5)}{10}}$, $j=1,2,3$, $\pi_j=-1.2\, \mathrm{e}^{\frac{\pi \mathrm{i} (j+5)}{10}}$, $j=4,5,6$, $\pi_j=1.2\, \mathrm{e}^{-\frac{\pi \mathrm{i} (j+5)}{10}}$, $j=7,\dots,10$  are simple poles. The function $f$ is analytic in the annulus $\mathcal{A}$ with $r=0.9$ and $R=1.2$ and has in general 40 pole-type singularities, 30 of which are located in the annulus $\Tilde{\mathcal{A}}^{(-)}$ with $\Tilde{r}<0.9$ and the other 10 in the annulus $\Tilde{\mathcal{A}}^{(+)}$ with $\Tilde{R}>1.2$ (see Figure \ref{figan}).  

\begin{figure}[h!]
  \centering
    \includegraphics[width=0.55\linewidth]{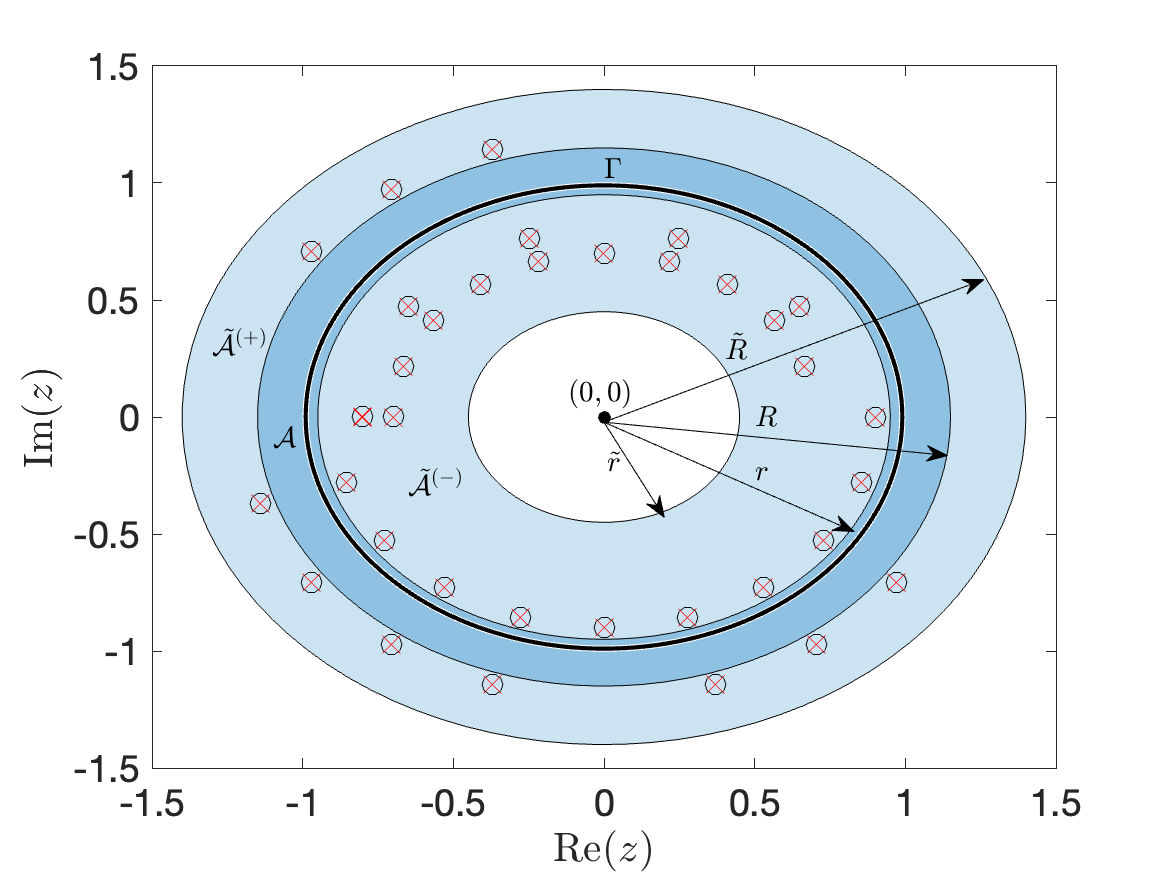}
      \caption{ Annulus $\mathcal{A}$ (dark blue region) in which a function $f$ given by  (\ref{exfun}) is analytic, annuli $\Tilde{\mathcal{A}}^{(-)}$ and   $\Tilde{\mathcal{A}}^{(+)}$ in which $f$ has  meromorphic extension, the contour $\Gamma$,  poles of a function $f$ as in (\ref{exfun})  (black circles) and their estimated location (red crosses) computed by formulas (\ref{sin}) and (\ref{sin1}) applying Algorithm \ref{alg3}.}
  \label{figan}
\end{figure}

We construct an approximation to $f$ in the form (\ref{nn}) by  Algorithm \ref{alg3}. As input data we use $300$ function values $f(\rho \mathrm{e}^{\frac{2\pi \mathrm{i} j}{300} })$, $j=0,\dots,299$, i.e.  $n=150$, located on the contour  $\Gamma$ as in (\ref{count}) with $\rho=0.99$.  We choose parameters  $N_1^{(\pm)}=70$ and $M_1^{(\pm)}=70$. Algorithm \ref{alg1} reduces  the number of the neurons of the hidden layer of the component $\Phi^{(-)}$ from $M_1^{(-)}=70$ to $M^{(-)}=30$, and of the component $\Phi^{(+)}$ from $M_1^{(+)}=70$ to $M^{(+)}=10$, and finds the optimal parameters that ensure good approximation $N^{(-)}=30$ and $N^{(+)}=10$. In case of the component $\Phi^{(+)}$, the Step 2 of Algorithm \ref{alg1} is performed 6 times. The parameters $N_1^{(+)}=70$ and $M_1^{(+)}=70$ are first reduced to $N_1^{(+)}=28$ and $M_1^{(+)}=28$, in the next four times  to 20, 15, 12 and 11, respectively, and finally at the 6th iteration to $N^{(+)}=10$ and $M^{(+)}=10$.  For the component $\Phi^{(-)}$, the Step 2 of Algorithm \ref{alg1} is performed only a single time. 

\begin{figure}[h!]
  \centering
  \begin{subfigure}[b]{0.31\linewidth}
    \includegraphics[width=1\linewidth]{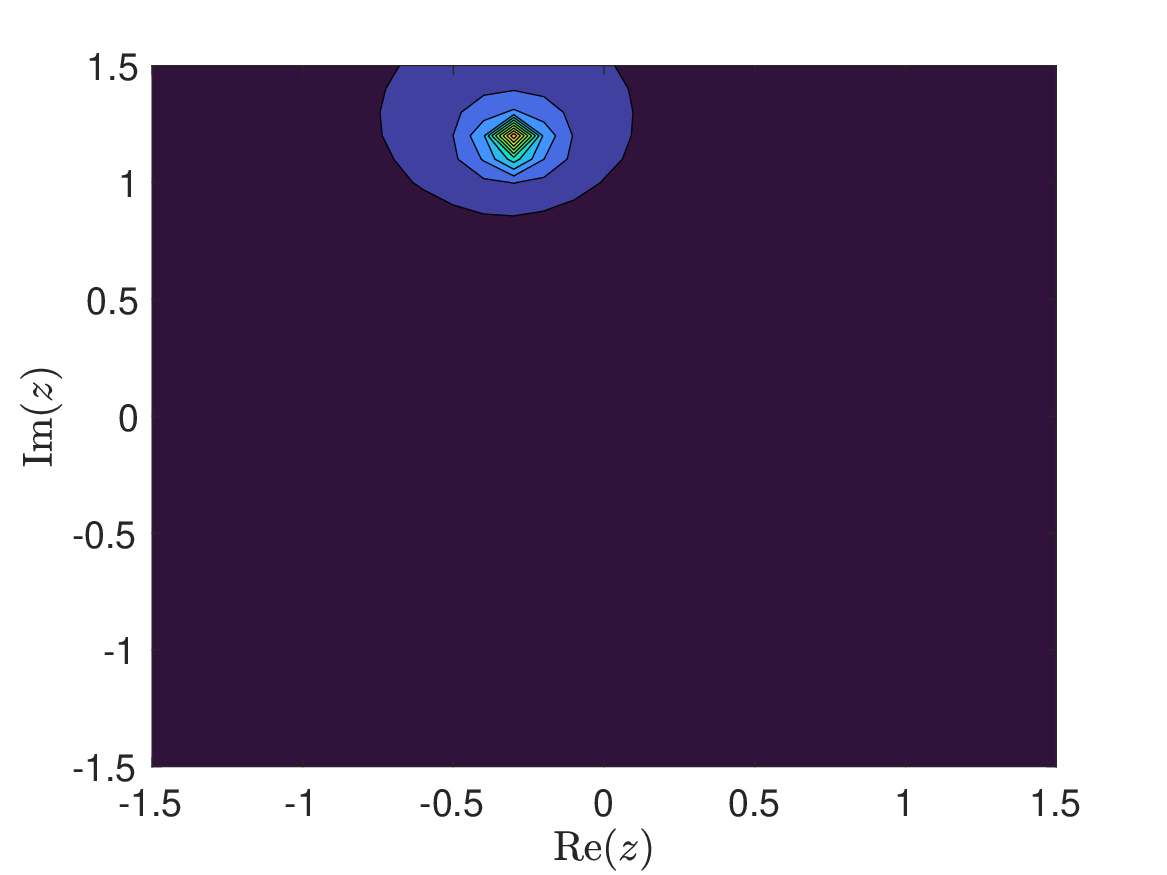}
  \end{subfigure}
  \begin{subfigure}[b]{0.31\linewidth}
    \includegraphics[width=1\linewidth]{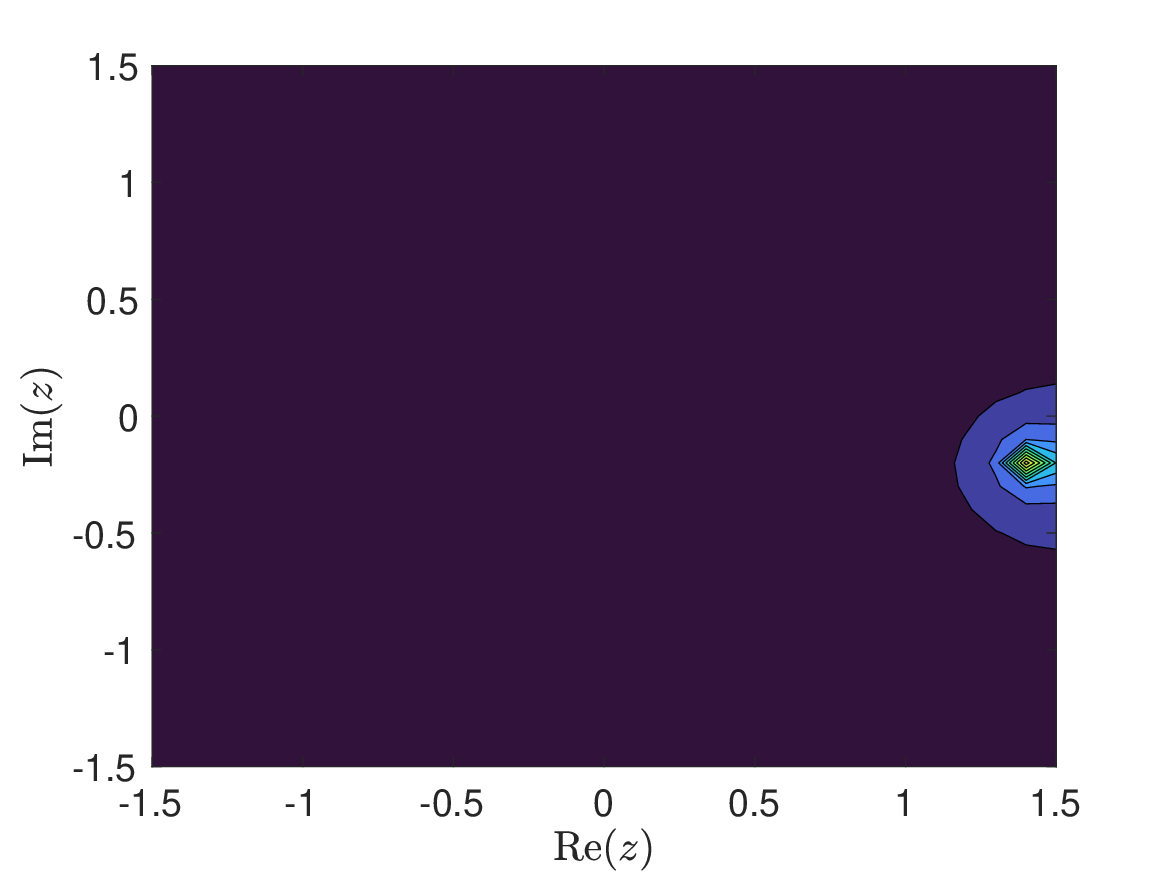}
  \end{subfigure}
    \begin{subfigure}[b]{0.31\linewidth}
    \includegraphics[width=1\linewidth]{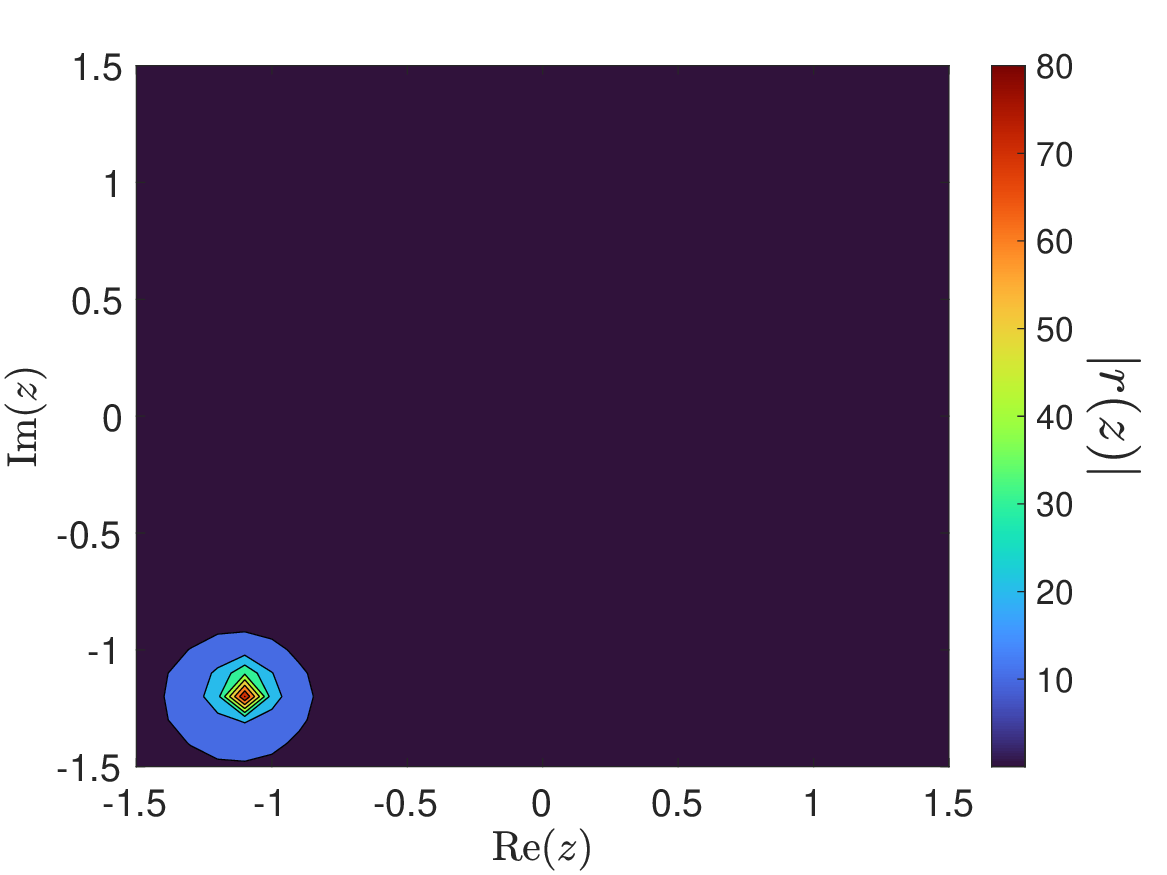}
  \end{subfigure}
  \caption{Activation functions $r$ with a pole $z_0=-0.3+ \frac{35}{30} \mathrm{i}$ (left), $z_0=1.43- 0.2 \mathrm{i}$  (middle), and $z_0=-1.1- \frac{7}{6} \mathrm{i}$  (right) employed in Subsection  \ref{ex1}.}
  \label{figac1}
\end{figure}

For both components $\Phi^{(\pm)}$, we use the same activation function $r^{(\pm)}=r$ that is computed as the Laurent-Pad\'{e} approximation of type $(1,1)$ to $\omega(z)=\frac{\cos z}{z-z_0}$ by  Algorithm \ref{alg5}. 
We now consider three  different experiments. For each of these three cases, we allow a random choice of $C_{k 0}^{(-)}$, $k=1,\dots,29$, and $C_{k 0}^{(+)}$, $k=1,\dots,9$ from the domain $[-1,-0.5]\times [0.5,1] \, \mathrm{i}$, and choose different locations of the pole $z_0$ of the activation function $r$:
\begin{itemize}
    \item \textbf{Location 1:} $z_0=-0.3+ \frac{35}{30} \mathrm{i}$ (Figure \ref{figac1}, left);  
        \item \textbf{Location 2:}    $z_0=1.43- 0.2 \mathrm{i}$ (Figure \ref{figac1}, middle); 
         \item \textbf{Location 3:}    $z_0=-1.1- \frac{7}{6} \mathrm{i}$ (Figure \ref{figac1}, right).  
\end{itemize}

In Figure \ref{figan}, we demonstrate the accuracy of the recovery of poles of $f$ as in (\ref{exfun}) by Algorithm \ref{alg3} for the pole $z_0$ at Location 1 (for Locations 2 and 3 there is no difference in the results).     Figure \ref{figweight}  shows the corresponding location of weights $w_{k 1}^{(-)}$, $k=1,\dots,30$ and $w_{k 1}^{(+)}$, $k=1,\dots,10$ for all three cases.
Note that for each  pole $\zeta_j$ with double multiplicity, we get two different neurons of the hidden layer of the neural network component $\Phi^{(-)}$. Thus, in Figure \ref{figan} we observe 35 pairwise distinct poles of $f$ and in Figure \ref{figweight}, 40 weights that correspond to 40 neurons of $\Phi$.

\begin{figure}[h!]
  \centering
  \begin{subfigure}[b]{0.31\linewidth}
    \includegraphics[width=1\linewidth]{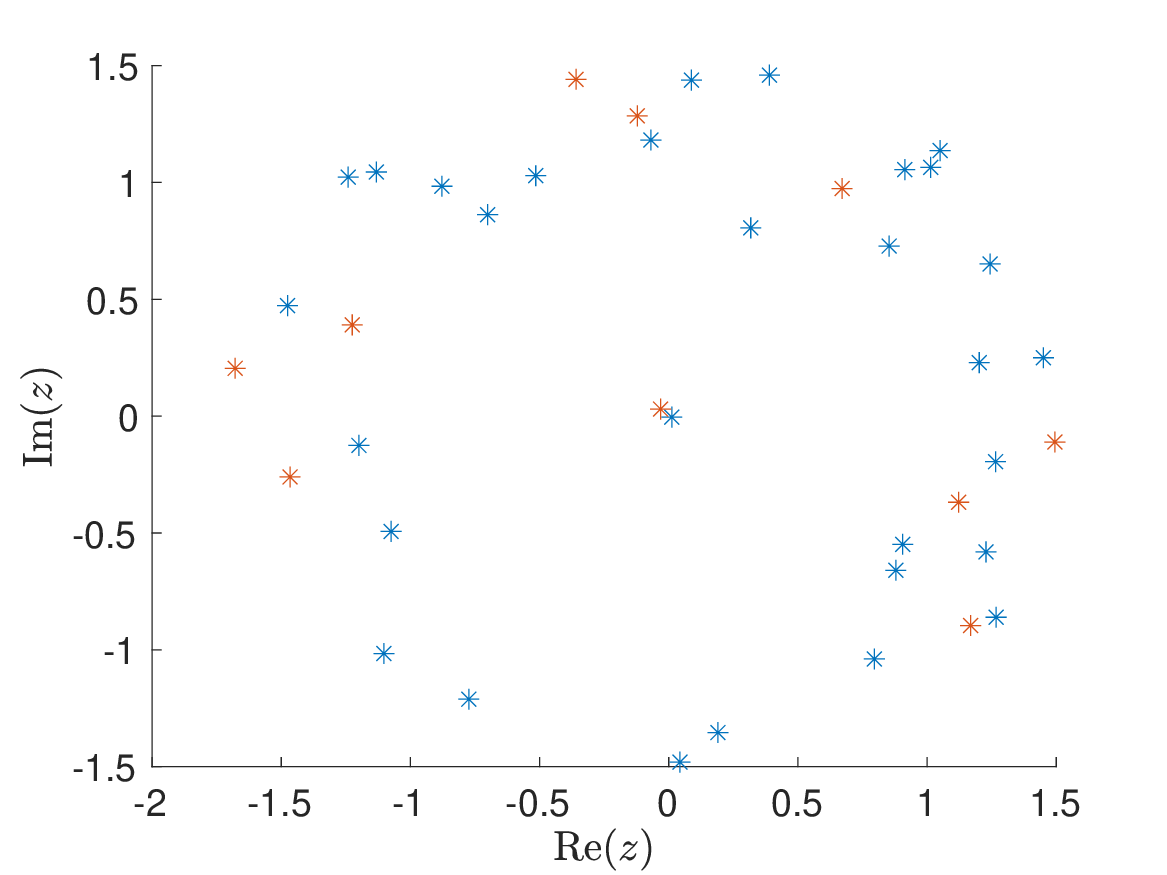}
  \end{subfigure}
  \begin{subfigure}[b]{0.31\linewidth}
    \includegraphics[width=1\linewidth]{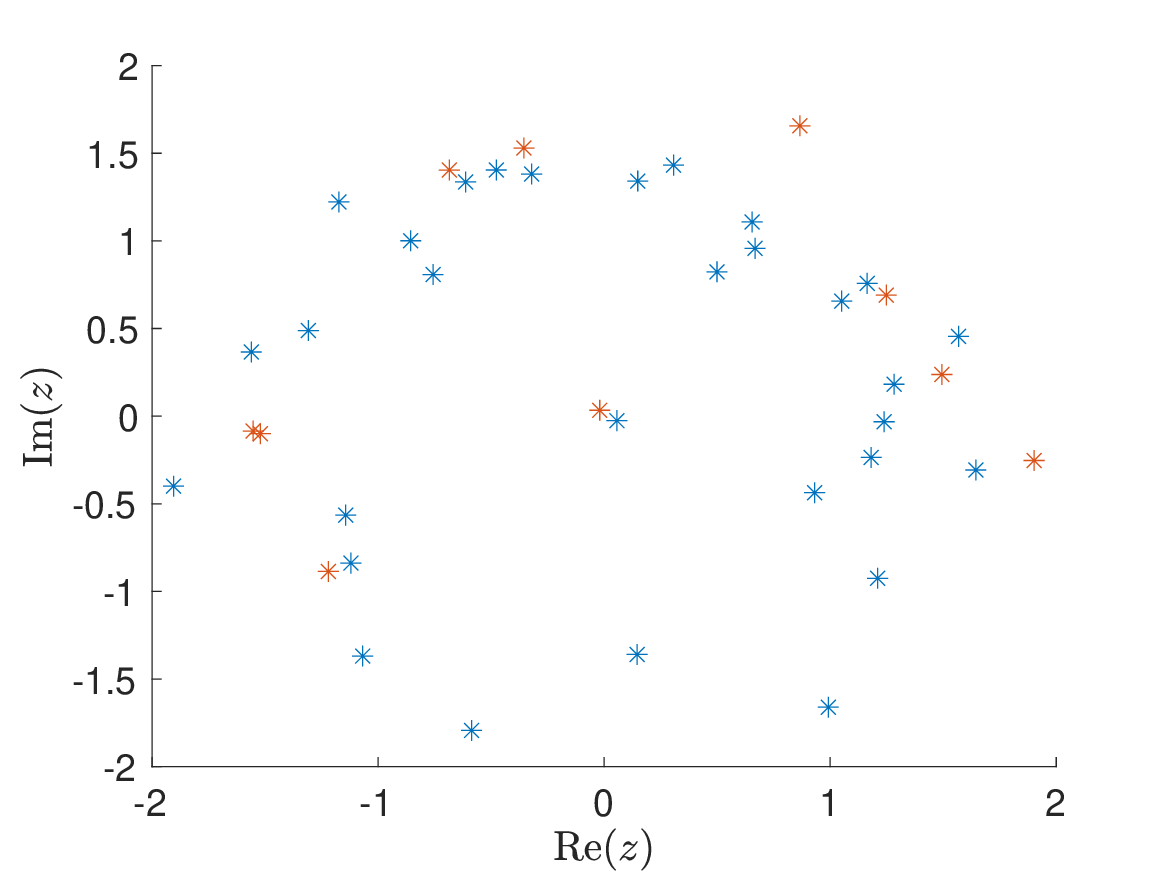}
  \end{subfigure}
  \begin{subfigure}[b]{0.31\linewidth}
    \includegraphics[width=1\linewidth]{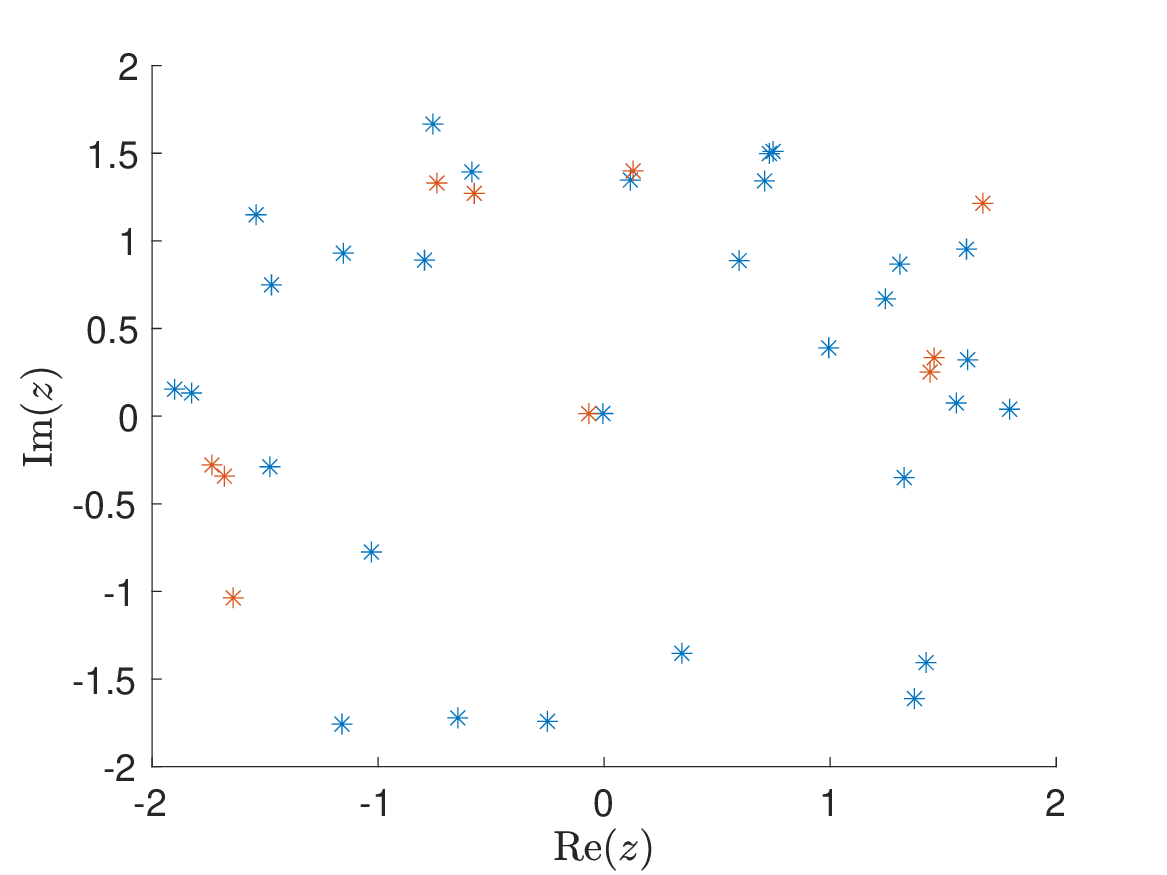}
  \end{subfigure}
  \caption{Weights $w_{k1}^{(-)}$, $k=1,\dots,30$ (blue) and $w_{k1}^{(+)}$, $k=1,\dots,10$ (orange) computed for Location 1 (left), Location 2 (middle) and Location 3 (right) of the pole $z_0$.}
  \label{figweight}
\end{figure}

\subsection{Construction of extensions of the time-dependent solutions of nonlinear PDEs and study  the dynamics of their singularities}

In this subsection, we describe the application of our method to the construction of  extensions of the time-dependent solutions of nonlinear PDEs into the complex plane and study  the dynamics of their singularities. Namely,  we consider  the following equation, 
\begin{equation}\label{pde1}
    v_t-\nu v_{xx}+(H(v)v)_x=0, 
\end{equation}
where $\nu>0$ and $H$ is the Hilbert transform. The equation (\ref{pde1}) is an analog to the conservative form of Burgers' equation and was considered in \cite{BM96,W03}. The explicit solution $v(x,t): \rr\times[0,1]\rightarrow \rr$ to this equation is known and is given by 
\begin{equation}\label{exsol}
    v(x,t)=\eta+\frac{\nu(1-\beta^2 \mathrm{e}^{2\eta t})}{1+\beta^2 \mathrm{e}^{2\eta t}-2 \beta \, \mathrm{e}^{\eta t} \cos x},
\end{equation}
where $\eta, \beta>0$ are arbitrary constants. Formally substituting a complex variable $z$ instead of the real variable $x$ in   (\ref{exsol}), we obtain an extension of the solution (\ref{exsol})   to the complex plane, which we denote by $u(z,t): \cc\times[0,1]\rightarrow \cc$. 
The singularities of the extended solution are simple poles located at 
$$
    z(t)=2 \ell \pi \pm \mathrm{i}(\eta t + \ln \beta), \quad \ell=0,\pm 1, \pm 2, \dots.
    $$
When $\beta<1$, these singularities move toward the real axis.  They  reach the real axis at the time point $t=-\ln \beta / \eta$ at which occurs the blow up of the extended solution.


Similarly, as in \cite{W03}, we consider a solution to (\ref{pde1}) that is $2\pi$-periodic in $x$. Using the Fourier spectral method, we construct the corresponding solution in $[-\pi,\pi]$ as truncated Fourier series
\begin{equation}\label{appsr}
    v(x,t)\approx \sum\limits_{k=-n}^n a_k(t) \, \mathrm{e}^{\mathrm{i} k x},
\end{equation}
for some $n \in \nn$.
The Fourier coefficients $a_k(0)$ are computed using the initial condition values at $2 n$ points $v\left( \frac{2 \pi j}{2 n} ,0 \right)$, $j=0,\dots,2 n-1$, applying FFT. For other time points $t \in (0,1]$,  the coefficients $a_k(t)$ are computed as the solution to the system of ODEs
$$
 \frac{\mathrm{d} a_k(t)}{\mathrm{d} t } + \nu k^2 a_k(t) - k \sum\limits_{ \substack{j+\ell=k \\ |j|,|\ell|\leq n}} \mathrm{sign}(j) a_j(t) a_{\ell}(t)=0, \quad k=-n,\dots,n.
 $$
Further we discuss a method to construct an extension of the solution from the real line to the complex plane using the modified neural network ideas from Section \ref{spt}. We proceed as follows.  First, we represent the extended solution in the following form,
$$
    u(z,t)\approx \sum\limits_{k=-n}^n a_k(t) \, \mathrm{e}^{\mathrm{i} k z}=  u^{(-)}(z,t)+u^{(+)}(z,t),
    $$
where  $u^{(-)}(z,t):=\frac{a_0(t)}{2}+\sum_{k=1}^{n} a_{-k}(t) \mathrm{e}^{-\mathrm{i} k z}$   and $u^{(+)}(z,t):=\frac{a_0(t)}{2}+\sum_{k=1}^n a_k(t) \mathrm{e}^{\mathrm{i} k z}$. Next, functions $u^{(\pm)}(z,t)$ are approximated by the corresponding neural network components
\begin{align}
   \Phi^{(+)}(\mathrm{e}^{\mathrm{i} z},t) & =  \mathbf{W}_2^{(+)}(t) r^{(+)} \left( \mathbf{W}_1^{(+)}(t) \mathrm{e}^{\mathrm{i}  z} - \boldsymbol{b}_1^{(+)}(t)  \right) - b_2^{(+)}(t), \label{f1t} \\
      \Phi^{(-)}(\mathrm{e}^{\mathrm{i} z},t) & =  \mathbf{W}_2^{(-)}(t) r^{(-)} \left( \mathbf{W}_1^{(-)} (t)\mathrm{e}^{-\mathrm{i} z} - \boldsymbol{b}_1^{(-)} (t) \right) - b_2^{(-)}(t), \label{f2t}
\end{align}
and we get the final formula for the extended solution in $\cc$
\begin{equation}\label{appsnn}
    u(z,t)\approx \Phi(\mathrm{e}^{\mathrm{i} z},t)=  \Phi^{(+)}(\mathrm{e}^{\mathrm{i} z},t)+\Phi^{(-)}(\mathrm{e}^{\mathrm{i} z},t).
\end{equation}
The components $\Phi^{(+)}$ and $ \Phi^{(-)}$ are constructed using ideas similar to those employed in Algorithm \ref{alg3}, but 
replacing Laurent coefficients by Fourier coefficients $a_k(t)$ and taking into account their dependency on time.  The component $\Phi^{(+)}$ now detects the locations of the singularities according to the formula
\begin{equation}\label{spl}
\mathrm{e}^{\mathrm{i} s_{\ell}^{(+)}(t)}=  \frac{b^{(+)}_{\ell 1}(t)+z_0^{(+)} }{w^{(+)}_{\ell 1}(t)}  \quad \text{or }   \quad 
 s_{\ell}^{(+)}(t)= -\mathrm{i}  \ln \left(  \frac{b^{(+)}_{\ell 1}(t)+z_0^{(+)} }{w^{(+)}_{\ell 1}(t)} \right),
\end{equation}
such that $\mathrm{Im} \, s_{\ell}^{(+)}(t)<0$ for $\ell=1,\dots,M^{(+)}$ at given time point $t$. Correspondingly,  the component $\Phi^{(-)}$ detects the locations of the singularities according to the formula 
\begin{equation}\label{spl1}
 \mathrm{e}^{\mathrm{i} s_{\ell}^{(-)}(t)}=\frac{w^{(-)}_{\ell 1}(t)}{b^{(-)}_{\ell 1}(t)+z_0^{(-)}}  \quad \text{or }   \quad      s_{\ell}^{(-)}(t)= -\mathrm{i}  \ln \left(  \frac{w^{(-)}_{\ell 1}(t)}{b^{(-)}_{\ell 1}(t)+z_0^{(-)}} \right), 
\end{equation}
such that $\mathrm{Im} \, s_{\ell}^{(-)}(t)>0$ for $\ell=1,\dots,M^{(-)}$ at given time point $t$.  Note that for this modified method, the weights $w^{(\pm)}_{\ell 1}(t)$  and biases    $b^{(\pm)}_{\ell 1}(t)$ are carrying information about $\mathrm{e}^{\mathrm{i}  s_{\ell}^{(\pm)}(t)}$ (and thus not directly about $s_{\ell}^{(\pm)}(t)$). 

\begin{figure}[h!]
  \centering
  \begin{subfigure}[b]{0.31\linewidth}
    \includegraphics[width=1\linewidth]{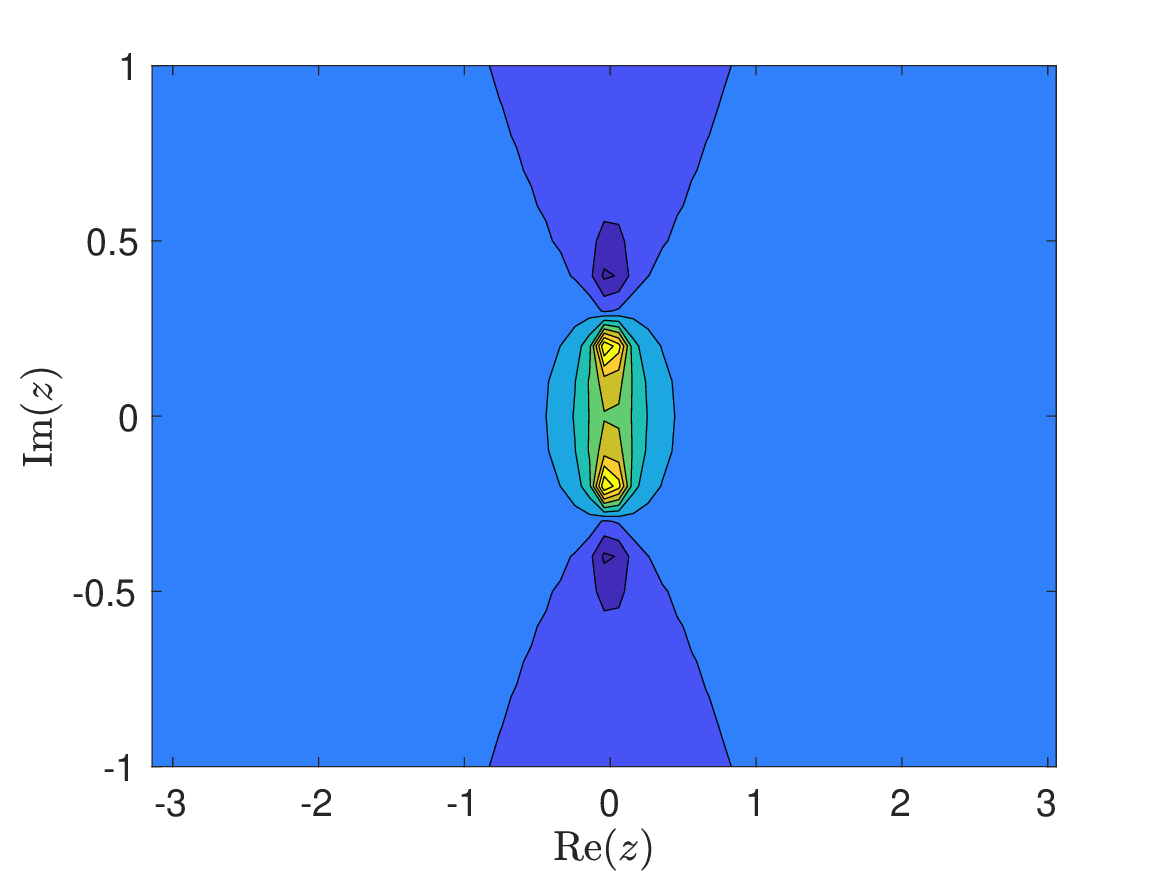}
  \end{subfigure}
  \begin{subfigure}[b]{0.31\linewidth}
    \includegraphics[width=1\linewidth]{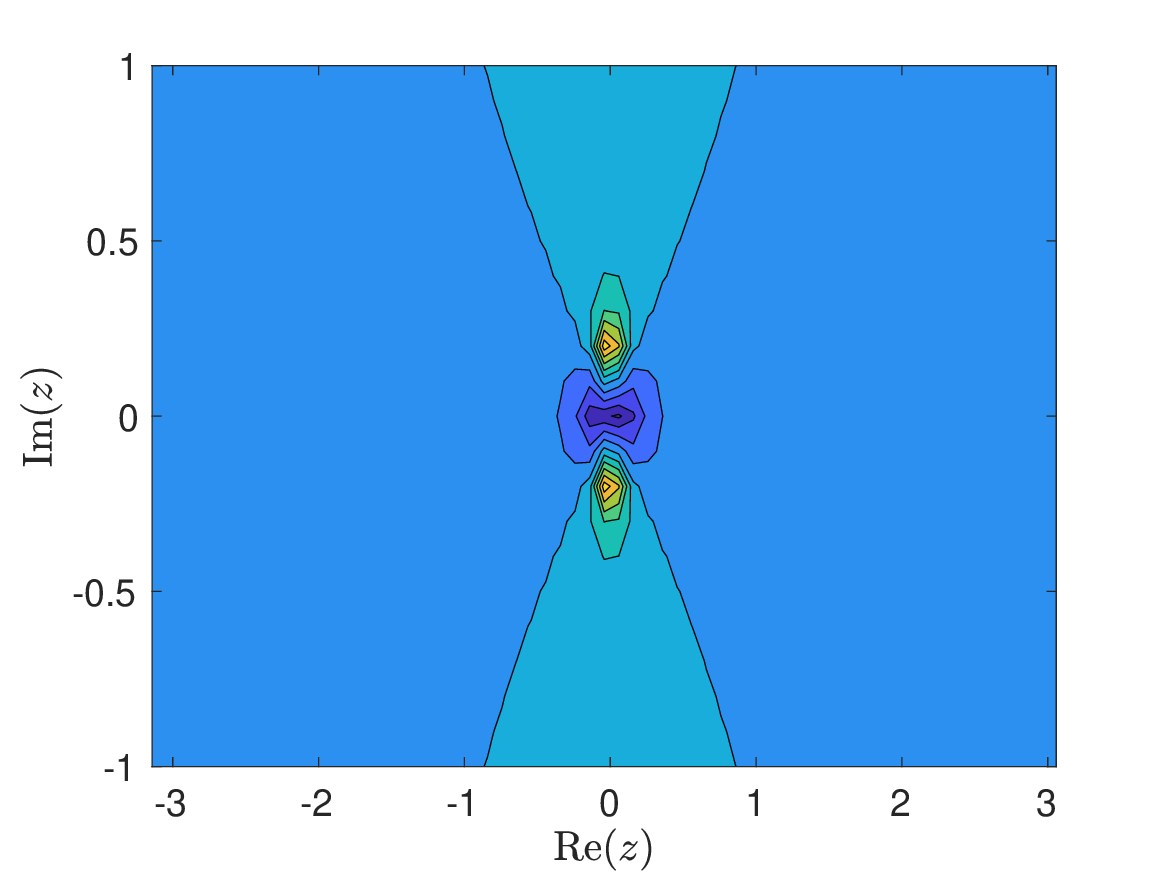}
  \end{subfigure}
  \begin{subfigure}[b]{0.31\linewidth}
    \includegraphics[width=1\linewidth]{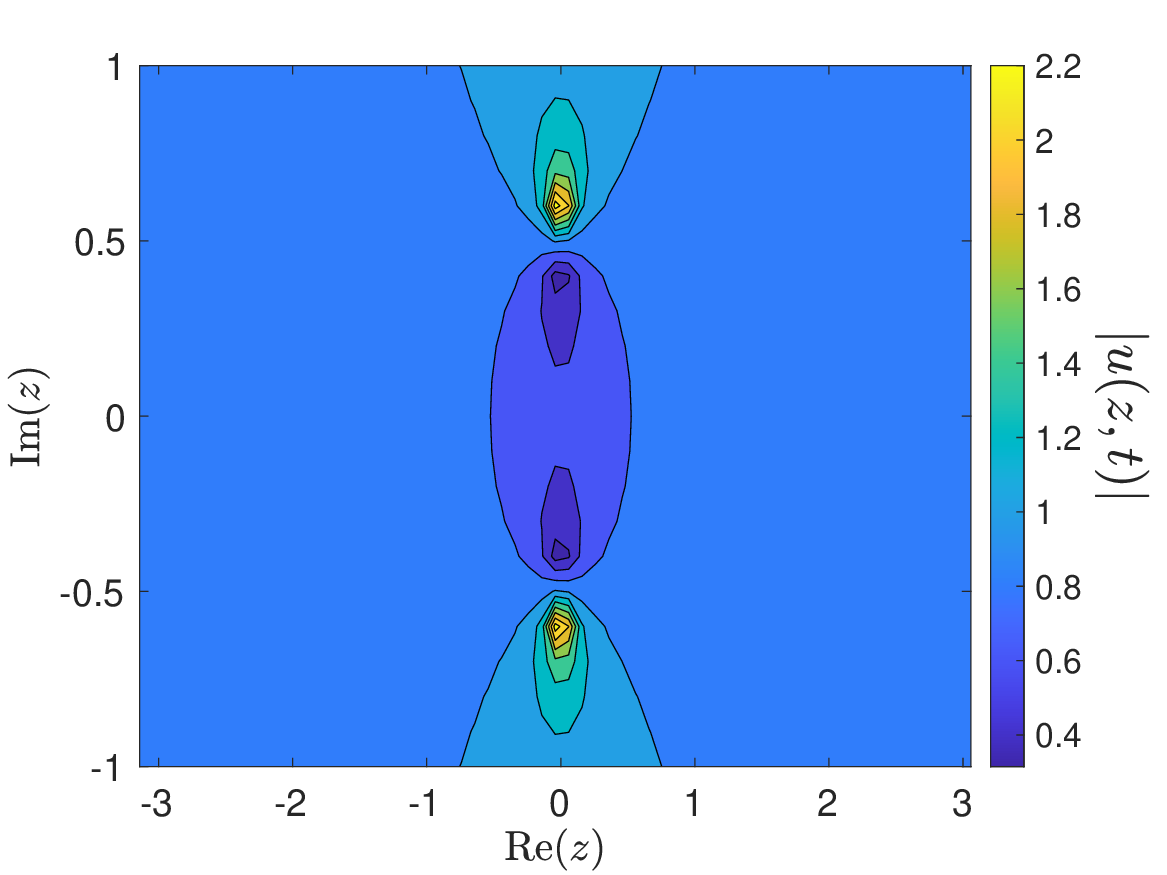}
  \end{subfigure}
  \begin{subfigure}[b]{0.31\linewidth}
    \includegraphics[width=1\linewidth]{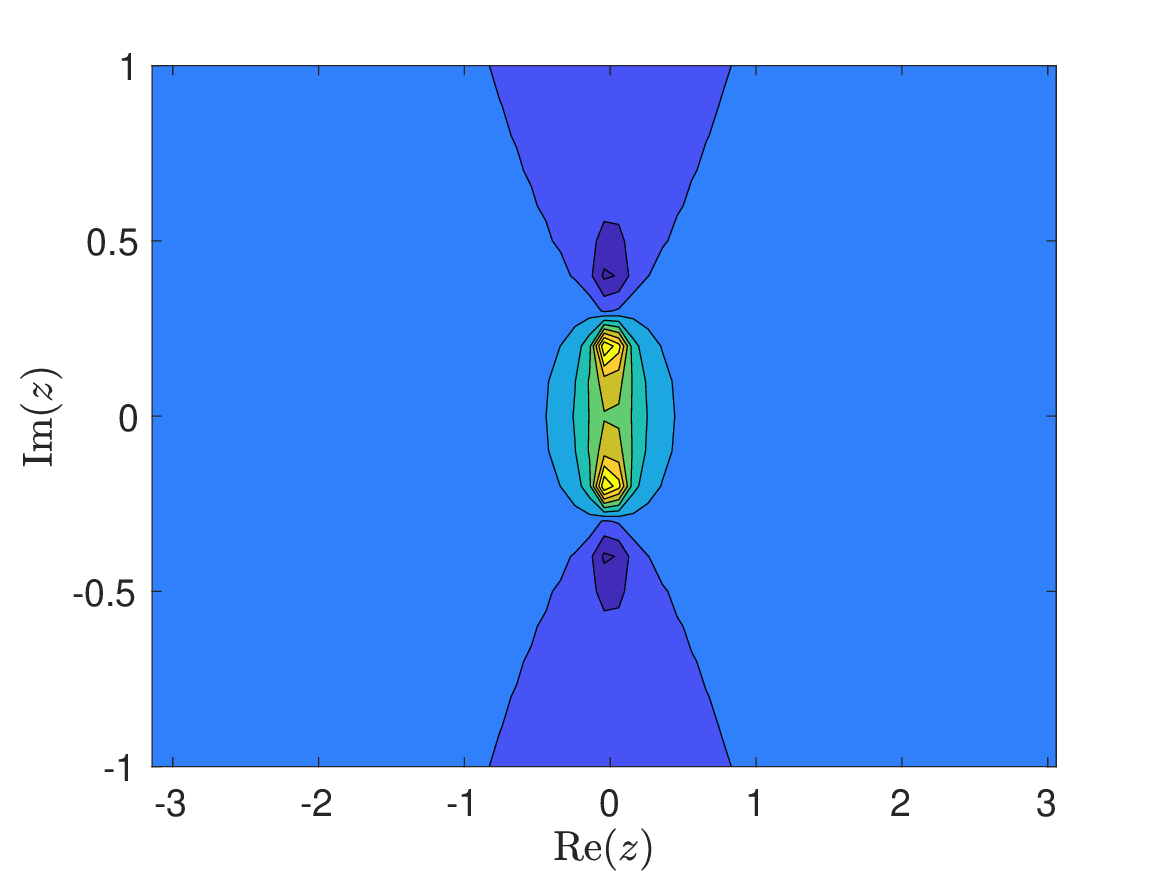}
    \caption{$t=0$}
  \end{subfigure}
  \begin{subfigure}[b]{0.31\linewidth}
    \includegraphics[width=1\linewidth]{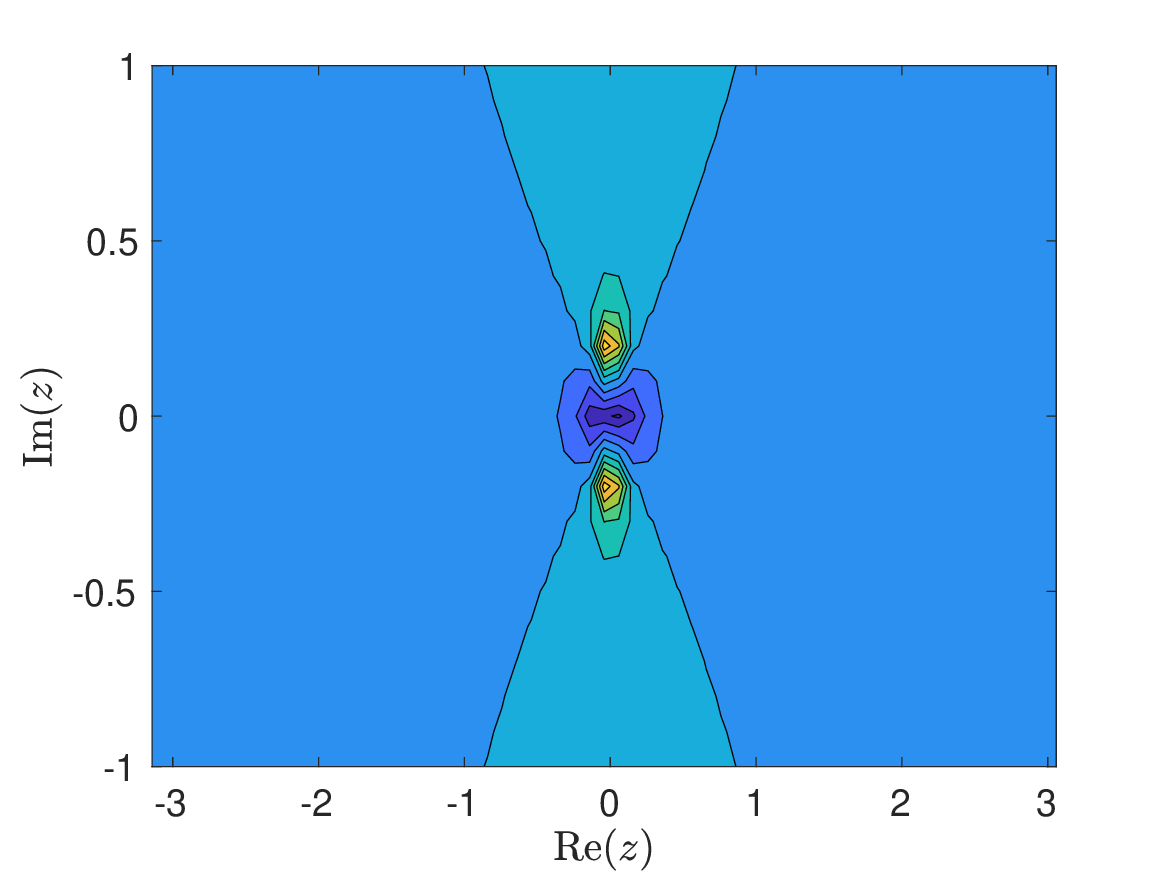}
    \caption{$t=0.4$}
  \end{subfigure}
  \begin{subfigure}[b]{0.31\linewidth}
    \includegraphics[width=1\linewidth]{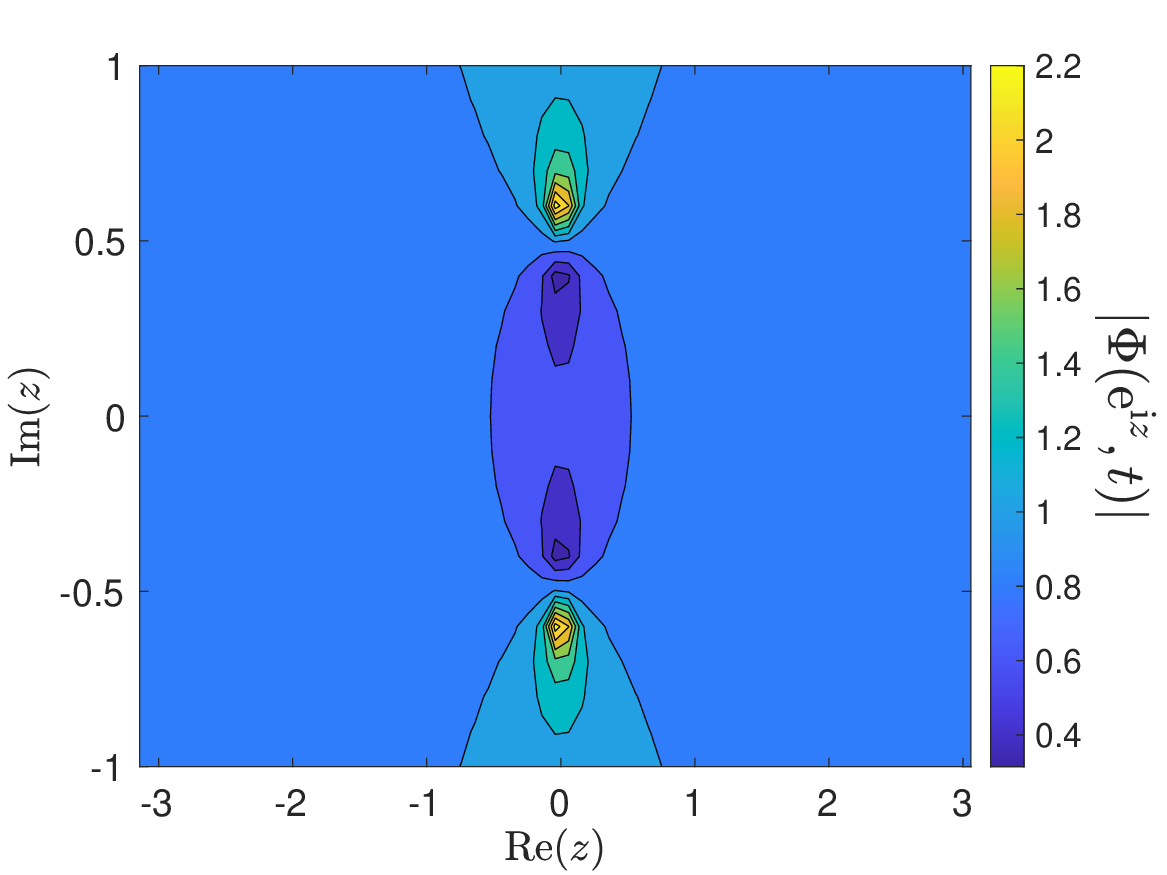}
    \caption{$t=0.8$}
  \end{subfigure}
  \caption{Extended solution $u(z,t)$ (first row) and its approximation $\Phi(\mathrm{e}^{\mathrm{i} z},t)$ as in (\ref{appsnn}) (second row)  for different time points $t=0, \, 0.4, \, 0.8$. The two poles of the extended solution cross in the real line at the time point $t=0.25$.}
  \label{fig1}
\end{figure}

We consider the PDE (\ref{pde1}) with parameters $\eta=1$, $\nu=0.1$ and $\beta=\mathrm{e}^{-1/4}$, and investigate the extended solution $u(z,t)$ for  $z \in [-\pi, \pi]\times  \mathrm{i} [-1, 1]$ and $t \in [0,1]$. In this domain, we then have two singularities of the form $z_1(t)=\mathrm{i}( t - \frac{1}{4})$ and $z_2(t)=-\mathrm{i}( t - \frac{1}{4})$, and the blow-up occurs in the real line at the time point $t= \frac{1}{4}$. We estimate the locations of singularities by  (\ref{spl})--(\ref{spl1}) and construct the extended solution as in  (\ref{appsnn}).  We employ parameters $n=40$, $\rho=0.99$, $N_1^{(\pm)}=M_1^{(\pm)}=10$. 
For each considered time point $t$, the method gives us the optimal values $N^{(\pm)}=M^{(\pm)}=1$. Note that we must set $\mathrm{tol}=10^{-3}$ in order to apply Algorithm \ref{alg1} to determine the correct parameters $M^{(\pm)}=1$ for each time point $t$, especially those close to the blow-up.
We  use again the same activation function  $r^{(\pm)}=r$  for both neural network components $\Phi^{(\pm)}$ that is constructed as the Laurent-Pad\'{e} approximation of type $(1,1)$ to $\omega(z)=\frac{\cos z}{z-z_0}$ with $z_0=-1.2$ by Algorithm \ref{alg5}. 
In Figure \ref{fig1}, we present a comparison of the explicit extended solution $u(z,t)$ (as in (\ref{exsol}) but with a complex $z$) and its approximation $\Phi(\mathrm{e}^{\mathrm{i} z},t)$ as in (\ref{appsnn})  for  time points $t=0, \, 0.4, \, 0.8$. Figure \ref{fig1} also allows us to observe the dynamics of the complex singularities $z_1(t)$ and $z_2(t)$. 

\begin{figure}[h!]
  \centering
  \begin{subfigure}[b]{0.45\linewidth}
    \includegraphics[width=0.9\linewidth]{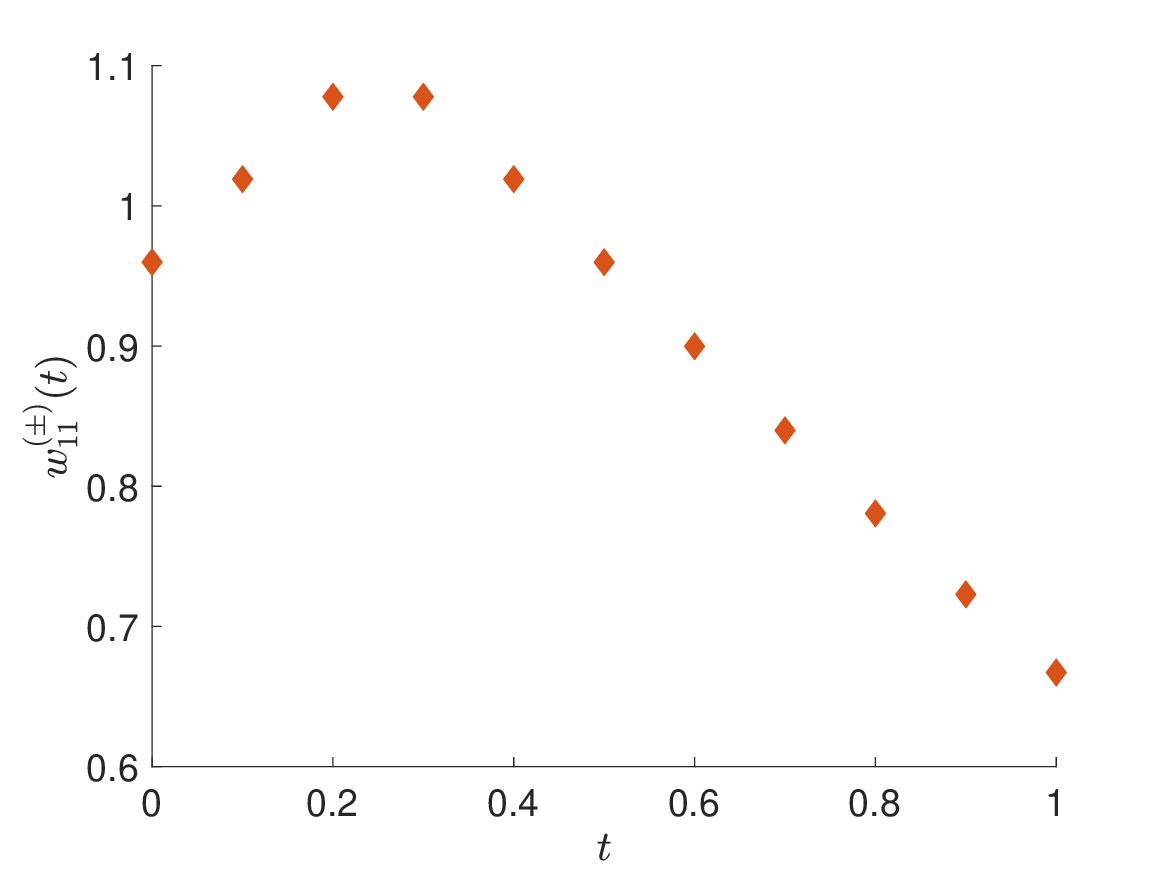}
  \end{subfigure}
  \begin{subfigure}[b]{0.45\linewidth}
    \includegraphics[width=0.9\linewidth]{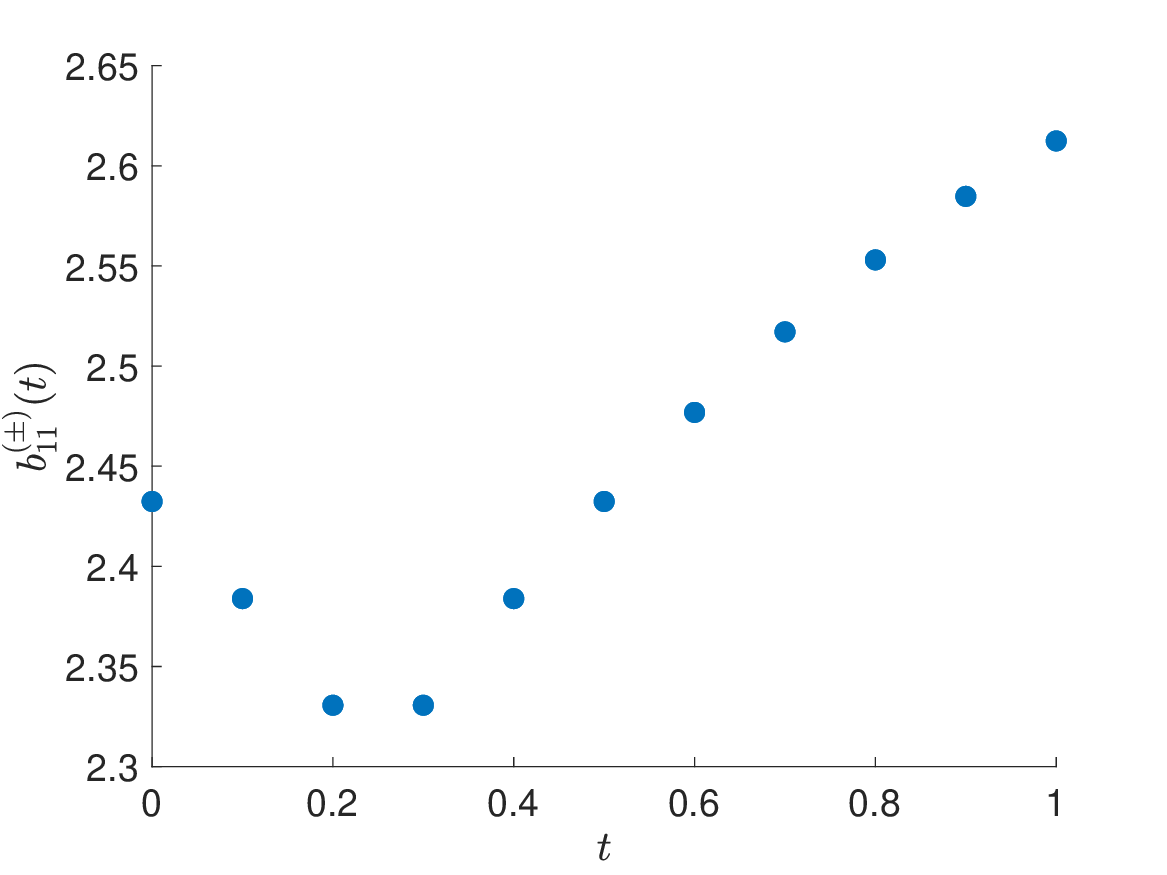}
  \end{subfigure}
  \caption{ Weights $w_{11}^{(\pm)}(t)$  (left) and biases $b_{11}^{(\pm)}(t)$  (right) for time points $t=0, 0.1, 0.2,\dots,1$ and a pole $z_0=-1.2$ of the activation function.}
  \label{fig2}
\end{figure}

Further, we are also interested to investigate the behavior of weights $\mathbf{W}_1^{(\pm)}(t)=w_{11}^{(\pm)}(t)$ and biases $\boldsymbol{b}_1^{(\pm)}(t)=b_{11}^{(\pm)}(t)$ of the hidden layers of the components $\Phi^{(\pm)}$ as in (\ref{f1t}) and (\ref{f2t}) with respect to time $t$.
The singularities $z_1(t)$ and $z_2(t)$ satisfy 
\begin{align*}
\mathrm{Im} (z_1(t)) &<0 \quad \text{ for } \quad 0\leq t <\frac{1}{4} \quad \text{ and } \quad \mathrm{Im} (z_1(t))>0 \quad \text{ for } \quad \frac{1}{4}< t \leq 1, \\
\mathrm{Im} (z_2(t)) &>0 \quad \text{ for } \quad 0\leq t <\frac{1}{4} \quad \text{ and } \quad \mathrm{Im} (z_2(t))<0 \quad \text{ for } \quad \frac{1}{4}< t \leq 1.
\end{align*}
Thus, the neural network component $\Phi^{(+)}$ detects the pole $z_1(t)$ for $0\leq t <\frac{1}{4}$ (i.e., before the blow-up), and the neural network component $\Phi^{(-)}$ detects the pole $z_1(t)$ for $\frac{1}{4}< t \leq 1$ (i.e., after the blow-up), and vice versa for the pole $z_2(t)$.  Consequently, the weight $w_{11}^{(+)}(t)$ and the bias $b_{11}^{(+)}(t)$ are carrying information about $\mathrm{e}^{\mathrm{i} z_1(t)}$ for $0\leq t <\frac{1}{4}$ ($\mathrm{e}^{\mathrm{i} z_2(t)}$ for $\frac{1}{4}< t \leq 1$), and  the weight $w_{11}^{(-)}(t)$ and the bias $b_{11}^{(-)}(t)$ are carrying information about $\mathrm{e}^{\mathrm{i} z_1(t)}$ for $\frac{1}{4}< t \leq 1$ ($\mathrm{e}^{\mathrm{i} z_2(t)}$ 
 for $0\leq t <\frac{1}{4}$). We obtain the following numerical errors of the  computation of singularities $z_1(t)$ and $z_2(t)$ by 
(\ref{spl})--(\ref{spl1}) for $t=0, \, 0.4, \, 0.8$ 
\begin{align*}
|z_1(0)-s_1^{(+)}(0)|=2.2 \cdot 10^{-9}, & \quad   |z_2(0)-s_1^{(-)}(0)|=2.2 \cdot 10^{-9}; \\
|z_1(0.4)-s_1^{(-)}(0.4)|=2.9 \cdot 10^{-6}, & \quad   |z_2(0.4)-s_1^{(+)}(0.4)|=2.9 \cdot 10^{-6}; \\
    |z_1(0.8)-s_1^{(-)}(0.8)|=3.74  \cdot 10^{-16}, & \quad   |z_2(0.8)-s_1^{(+)}(0.8)|=2.8 \cdot 10^{-16}.
\end{align*}

 In Figure \ref{fig2}, we present weights (left) and biases (right) for time points $t=0, 0.1, 0.2,\dots,1$ and fixed parameters $C_{10}^{(\pm)}=1$, which together with the application of the same activation function with a real pole $z_0=-1.2$ implies $w_{11}^{(+)}(t)=w_{11}^{(-)}(t) \in \rr$  and $b_{11}^{(+)}(t)=b_{11}^{(-)}(t) \in \rr$. We observe that weights $w_{11}^{(\pm)}(t)$ increase  and biases $b_{11}^{(\pm)}(t)$  decrease for $0\leq t < \frac{1}{4}$, i.e. before the blow up, and weights $w_{11}^{(\pm)}(t)$  decrease and biases $b_{11}^{(\pm)}(t)$ increase for $\frac{1}{4}< t \leq 1$, i.e. after the blow up. This behavior can be simply explained by noticing that 
$\mathrm{e}^{\mathrm{i} z_1(t)}=\mathrm{e}^{-(t-1/4)}$ decreases and $\mathrm{e}^{\mathrm{i} z_2(t)}=\mathrm{e}^{(t-1/4)}$  increases and taking into account formulas (\ref{spl}) and (\ref{spl1}).
 Note also that a different choice of $C_{10}^{(\pm)} \in \rr$ or a pole $z_0 \in \rr$ does not change the behavior (only the numerical values) of the weights and the biases.

\section{Conclusions and plans for  future work}

We introduce a novel neural network construction that is designed to approximate functions with singularities. 
  Adaptive construction of the activation functions, a new backpropagation-free approach for training weights and biases of the hidden layers, and least-squares fitting for computation weights and biases in the output layers ensure accurate estimation  the locations of singularities of the functions under consideration and approximation in the rest of the domain. The capturing of pole-type singularities of the meromorphic functions is of particular interest to us. The weights and biases of the hidden layers are used to scale and shift, respectively, the poles of the activation functions to the estimated location of the singularities (see Figure \ref{figmain}).  That also implies that each neuron in our construction is carrying information only about a single singularity (see Figure \ref{nns}).   

  In this paper, we focused only on  pole-type singularities.  Following similar ideas, we can extend the approach to a neural network for the detection of  square root singularities. As described in Remark 3.2, our main concept of ``one singularity, one neuron'' would then not be kept. Consequently, in order to detect singularities of other types, such as square root singularities (or higher order), logarithmic branch points, or shocks using neural network-based approaches, different neural network structures must be developed. 

\section*{Statements and Declarations}
\textbf{Competing Interests:} The authors declare that they have no competing interests.

\section*{Acknowledgement}  F.D. is supported by the DFG project 468830823, by the TUM Georg Nemetschek Institute, and associated to DFG-SPP-229. The work of I.K. was partially supported by the US Department of Energy.

\small

\begin{thebibliography}{99999}

\bibitem{B87}
 B. Adhemar, {\em Laurent series and their Padé approximations}, Birkhäuser Basel, 1987. https://doi.org/10.1007/978-3-0348-9306-0

\bibitem{BM96}
G. R. Baker, X. Li and A. C. Morlet, {\em Analytic structure of two 1D-transport equations with nonlocal ﬂuxes}, Physica D, 91 (1996), pp.~349--375.

\bibitem{swim}
 E. L. Bolager, I. Burak, Ch. Datar, Q. Sun and F. Dietrich, {\em Sampling weights of deep neural networks}, NeurIPS 2023, pp. ~63075--63116.

\bibitem{BNT20}
N. Boullé, Y. Nakatsukasa and A. Townsend,  {\em Rational neural networks}, NeurIPS 2020, pp. 14243 - 14253. 10.5555/3495724.3496918


\bibitem{BG24}
 V. Borulko and V. Gritsenko, {\em Exponential and rational models for radar signals}, DIPED 2024, pp. 155-159. 10.1109/DIPED63529.2024.10706157.

\bibitem{BR2015}
 C. Brezinski and M. Redivo-Zaglia,
{\em New representations of Padé, Padé-type, and partial Padé approximants},
J. Comput. Appl. Math.,
284
(2015),
pp. 69-77,
https://doi.org/10.1016/j.cam.2014.07.007.

\bibitem{BC2020}
 M. Briani, A. Cuyt, F. Knaepkens and W.-sh. Lee, {\em VEXPA: Validated EXPonential Analysis through regular sub-sampling}, Signal Process., 177 (2022), 107722. https://doi.org/10.1016/j.sigpro.2020.107722


\bibitem{CC20}
 M. Capcelea and T. Capcelea, {\em Laurent-Padé approximation for locating singularities of meromorphic functions with values given on simple closed contours}, Bul. Acad. Ştiinţe Repub. Mold. Mat., \textbf{93} (2020), pp.~76--87.

 \bibitem{CL22}
 A. Caragea, D.G. Lee, J. Maly, G. Pfander and F. Voigtlaender, {\em  Quantitative approximation results for complex-valued neural networks}, SIAM J. Math. Data Sci., 4 (2022), pp.~553-–580.

\bibitem{CCL2018}
 Z. Chen, F. Chen, R. Lai, X. Zhang and C. Lu, {\em Rational neural networks for approximating graph convolution operator on jump discontinuities},  ICDM 2018, pp.~59--68. 10.1109/ICDM.2018.00021.

\bibitem{DD24}
 Ch. Datar, T. Kapoor, A. Chandra,Q. Sun, I. Burak, E. L. Bolager, A. Veselovska, M. Fornasier and F. Dietrich, {\em Solving partial differential equations with sampled neural networks}, ArXiv preprint: https://arxiv.org/abs/2405.20836



\bibitem{DPP21}
N.~Derevianko,  G.~Plonka and  M.~Petz,  
{\em From ESPRIT to ESPIRA: Estimation of Signal Parameters by Iterative Rational Approximation},  IMA J. Numer. Anal., 43 (2023), pp.~789--827.

\bibitem{DH25}
 N. Derevianko and L. Hübner, 
{\em Parameter estimation for multivariate exponential sums via iterative rational approximation},  ArXiv preprint: https://arxiv.org/abs/2504.19157

\bibitem{EF2015}
 J. Eggers and M. A. Fontelos, {\em Singularities: formation, structure, and propagation}, Cambridge Texts in Applied Mathematics. Cambridge University Press, 2015.


\bibitem{FH19}
 M. Fasondini, N. Hale, R. Spoerer and J.A.C. Weideman, {\em Quadratic Padé approximation: numerical aspects and applications}, Comp. Res. Model., 11 (2019), pp.~1017--1031.

\bibitem{GP2000}
 A. Goldberger, et al., {\em Physiobank, physiotoolkit, and physionet: Components of a new research resource for
complex physiologic signals}, 2000.

\bibitem{GGT13}
P. Gonnet,  S. G\"{u}ttel and L.N. Trefethen,   {\em Robust Padé approximation via SVD},   SIAM Review,  55 (2013), pp.~101--117.
 https://doi.org/10.1137/110853236

\bibitem{L2018}
 K. Guo and D. Labate,  {\em Detection of singularities by discrete multiscale directional representations}, J. Geom. Anal.  28 (2018), pp. 2102–2128. https://doi.org/10.1007/s12220-017-9897-x

\bibitem{JK2020}
 A.D. Jagtap, K. Kawaguchi and  G. Em. Karniadakis,  {\em Adaptive activation functions accelerate convergence in deep and physics-informed neural networks}, J. Comput. Phys., 404 (2020), 109136. https://doi.org/10.1016/j.jcp.2019.109136
 





\bibitem{K17}
 P. G. Kevrekidis, C. I. Siettos and  Y. G. Kevrekidis, {\em To infinity and some glimpses of beyond}, Nat. Commun.,
8:1562 (2017). https://doi.org/10.1038/s41467-017-01502-7.


\bibitem{L22}
 C. Lee, H. Hasegawa and S. Gao, {\em Complex-Valued Neural Networks: A Comprehensive Survey}, in IEEE CAA J. Autom. Sin., 9 (2022), pp.~1406--1426.



\bibitem{M92}
 S. Mallat and W. L. Hwang, {\em Singularity detection and processing with wavelets}, IEEE Trans. Inf. Theory.,  38 (1992), pp.~617--643. 






\bibitem{MSK}
 A. Molina, P. Schramowski and K. Kersting,
{\em Padé Activation Units: End-to-end learning of flexible activation functions in deep networks}, ICLR 2020.


\bibitem{P22}
V. Peiris, 
{\em Rational activation functions in neural network with uniform norm based loss function and its application in classification}, Commun. Optim. Theory., 2022 (2022), pp.~1--25.

\bibitem{R18}
M. Rizzardi, {\em Detection of the singularities of a complex function by numerical approximations of its Laurent coefficients}, Numer. Algor., 77 (2018), pp.~955--982. https://doi.org/10.1007/s11075-017-0349-2

\bibitem{KG15}
 K. Samiee, P. Kovács and M. Gabbouj, {\em Epileptic seizure classification of EEG time-series using rational discrete short-time Fourier transform}, IEEE Trans. Biomed. Eng. 62 (2015), pp.541-52. 10.1109/TBME.2014.2360101.


\bibitem{Tel17}
 M. Telgarsky,  {\em Neural networks and rational functions},  ICML'17, 70 (2017), pp. 3387–3393, 2017.







\bibitem{W03}
 J.A.C. Weideman, {\em Computing the dynamics of complex singularities of nonlinear PDEs}, SIAM J. Appl. Dyn. Syst., 2 (2003), pp.~171--186.

\bibitem{W22}
J.A.C. Weideman,  {\em Dynamics of complex singularities of nonlinear PDEs,} Recent Advances in Industrial and Applied Mathematics. SEMA SIMAI Springer Series, vol 1. Springer, Cham, 2022.

\end{thebibliography}

\end{document}